\documentclass[a4paper,12pt]{article}
 \usepackage[english]{babel}
  \usepackage{amsmath,amsfonts,amssymb,amsthm,relsize}
  \usepackage{hyperref} %%Warning: when you first run your tex
 \usepackage[active]{srcltx} %%allows you to jump from your editor
 \usepackage{color,soul} %%allows you to use the \st{...} command, which
\numberwithin{equation}{section}
\newtheorem{theorem}{Theorem}[section]

\newtheorem{lemma}[theorem]{Lemma}
\newtheorem{proposition}[theorem]{Proposition}

\theoremstyle{definition}
\newtheorem{definition}[theorem]{Definition}
\newtheorem{remark}{Remark}

\def\N{\mathbb{N}}

\def\R{\mathbb{R}}

\let\ol=\overline

\let\.=\cdot
\let\0=\emptyset

\def\O{\Omega}

\newcommand{\su}[2]{\genfrac{}{}{0pt}{}{#1}{#2}}

\def\thm#1{Theorem \ref{thm:#1}}

\newenvironment{formula}[1]{\begin{equation}\label{eq:#1}}
                       {\end{equation}\noindent}

\def\Fi#1{\begin{formula}{#1}}
\def\Ff{\end{formula}\noindent}

\setstcolor{red}

\begin{document}
\title{
{Persistence versus extinction under a climate change  in mixed environments}
}

\author{Hoang-Hung Vo \thanks{\textrm{Centre d'Analyse et de Math\'ematique Sociales, 190-198, Avenue de France 75244 Paris Cedex 13, France. Email:
}\ttfamily vhhungkhtn@gmail.com}}

\maketitle

\begin{abstract}
This paper is devoted to the study of the persistence versus extinction of species in the reaction-diffusion equation:

\begin{equation}
u_t-\Delta u=f({t},x_1-{c}t,y,u) \quad\quad \textrm{$t>0$, $x\in\O, $}\nonumber
\end{equation}
where $\O$ is of cylindrical type or partially periodic domain, $f$ is of Fisher-KPP type  and the scalar ${c}>0$ is a given forced speed. This type of equation originally comes from a model in population dynamics (see \cite{BDNZ},\cite{PL},\cite{SK}) to study the impact of climate change on the persistence versus extinction of species. From these works, we know that the dynamics is  governed by the traveling fronts $u(t,x_1,y)=U(x_1-ct,y)$, thus  characterizing the set of  traveling fronts plays a major role. In this paper, we first consider a more general model than the model of \cite{BDNZ} in higher dimensional space, where the environment is only assumed to be globally unfavorable  with  favorable pockets extending to infinity. We consider in two frameworks: the reaction term is time-independent or time-periodic dependent. For the latter, we  study the concentration of the species when the environment outside $\Omega$ becomes extremely unfavorable and further prove a symmetry breaking property of the fronts.

\end{abstract} 



\medskip

%``Anticipation of global change
%effects on environmental and health hazards associated with
%urticating forest insects'' of the program ``Biodiversit\'e''

\tableofcontents

%%%%%%%%%%%%%%%%%%%%%%%%%%%%%%%%%%%%%%%%%%%%%%%%%%%%%%%%%%%%%%%%%%

\section{Introduction and main results}

\subsection{Introduction and definitions}
In a pioneering paper \cite{BDNZ}, Berestycki et al. studied the influence of climate change (global warming) on the population dynamics of biological species, who are strongly sensitive to temperature conditions. The authors proposed a mathematical model in $\R$, which is formulated as a reaction-diffusion equation with a forced speed $c$ : 
\begin{equation}
u_t-u_{xx}=f(x-ct,u)\quad\quad x\in\R,\label{0.00}
\end{equation}
where $u$ denotes population density of species and $c$ is the speed of the climate change. A typical $f$ considered in \cite{BDNZ} is 
\begin{equation}
f(x,s)=\left\{\begin{array}{ll}
-sm&\textrm{for $x<0$ and $x>L$}\\
sm'\left( 1-\dfrac{s}{K}\right)&\textrm{for $0\leq x\leq L$},\label{0.00new}
\end{array} \right.
\end{equation}
for some positive constants $m,m',L,K$. This nonlinearity expresses that the  environment  is  unfavorable outside a compact set $[0,L]$ and favorable inside. The higher dimensional versions with more general type of $f$ were  studied later in \cite{BR1}, \cite{BR2}. Beside that a similar model was also considered  in the context of competing species by Potapov and Lewis \cite{PL}, where the authors investigated the co-existence of two species under the effect of climate change and moving range boundaries on habitat invasibility. All these papers assume that the environments are  completely  unfavorable near infinity, i.e the favorable zone has compact support. More precisely, there exist $R, m>0$ such that
\begin{equation}
f_s(x,0)\leq-m,\quad\quad\forall|x|\geq R.\label{V10.1}
\end{equation}
Note that $f_s(x,0)$ is understood as the initial per capita rate of growth.

The main purpose of this paper is to study the criterion for persistence and extinction of species in  more general frameworks  than the ones considered in these previous works \cite{BDNZ},\cite{BR1},\cite{BR2},\cite{PL} and further provide some applications of this theory. Our aim is to deal with the new cases for which  condition (\ref{V10.1}) is  no longer true. We extend the model of (\ref{0.00}) in two frameworks. The first one is  for an infinite cylindrical domain with Neumann boundary condition:
\begin{equation}
\left\{\begin{array}{ll}
u_t-\Delta u=f(x_1-{c}t,y,u) & \textrm{$t>0$, $x\in\O $}\\
\partial_\nu u(t,x_1,y)=0 & \textrm{$t>0$, $x\in \partial\O$},\label{0}
\end{array} \right.
\end{equation} 
where $\Omega=\R\times\omega$, $\omega$ is an open bounded and smooth domain in $\R^{N-1}$, $\nu$ denotes the exterior unit normal vector field to $\O$. In this framework, the environments are assumed to be independent of time. We are especially interested in considering  environments of mixed type, which are only assumed to be globally unfavorable at infinity.

One can think of the environment containing both favorable and unfavorable regions that extend all the way to infinity, namely $f_s(x_1,y,0)>0$ and $f_s(x_1,y,0)<0$ respectively as $x_1\to\pm\infty$, depending on the location of $y$. The competitive and mutual influence  between these regions play a major role in characterizing the persistence and extinction of the species in the whole domain. 
% For example of such environments, we refer to the species in alpine ecosystem, their habitats are very complex, heterogeneous or fragmentary and the climate change may give different effects on each patch of their environment. Our paper presents a new look on such environments.

Mathematically, we will use a global condition in terms of spectral property  to describe that the environment is globally unfavorable at infinity. The more detailed explanations of this condition will be given in  subsection 1.2.1.

In the second framework, we investigate another type of mixed environment with periodic dependence on $y$ and $t$. More precisely, the equation is now of following type
\begin{equation}
u_t-\Delta u=f({t},x_1-{c}t,y,u) \quad\quad \textrm{$t>0$, $x=(x_1,y)\in\R^N $},\nonumber
\end{equation}
where the nonlinearity reaction is assumed to be periodic in $y$ and $t$. The time-periodic dependent reaction has been previously investigated in various frameworks, the interested reader are referred to \cite{BR1},\cite{DP1},\cite{NA1},\cite{NA2},\cite{PW1},\cite{SK}. The main difference of the present work with respect to these papers is that here $f$ is not assumed to be periodic in  $x_1$-direction  but be shifted  with the forced speed ${c}$, which can be seen as an effect of climate change. This has been considered in \cite{BR1} for an environment pointwise unfavorable at infinity. Our extension here is to consider $f_s(t,x_1,y,0)$ to be sign-changing depending on the location of $y\in\R^{N-1}$ at the time $t$. We only require $f$ to satisfy a global condition as $x_1\to\pm\infty$. The additional difficulties are due to the fact that we do not a priori require the solutions to be periodic in $y$ nor in $t$ and also we do not impose any boundary conditions as $x_1\to\pm\infty$. The time-periodic dependence of reaction term can be thought of as representation of a seasonal dependence of environment.

We further investigate the concentration of the species in  regions surrounded by highly hostile environments. More precisely, our aim is to describe the dynamics of the species in the first framework not only in the cylindrical domain $\O$ but in the whole space $\R^N$ under the assumption that the environment outside $\O$ becomes more and more unfavorable. From the biological point of view one may wonder whether the species still survives if some parts of the environment becomes extremely unfavorable. This question can be addressed by solving the following mathematical problem. We consider equation (\ref{0}) in the whole space $\R^N$ and study the limit of the sequence of traveling fronts with a reaction term $F_n(x,s)$ such that their growth rates are  negative outside the cylindrical  domain $\O$ and tend to $-\infty$ as $n\to\infty$. These solutions solve the equations 
$$\Delta U_n+c\partial_1 U_n+F_n(x,U_n)=0,\quad x\in\R^N,$$
where $F_n(x,s)=f(x,s)$ for $x\in\ol\O$ and $\dfrac{\partial F_n}{\partial s}(x,0)\to-\infty$ as $n\to\infty$ locally uniformly in $\R^N\setminus\O$. If the species survives, we aim to characterize the limit. This is the object of section 4. Very recently, Guo and Hamel \cite{GH} have studied the similar problem on the periodic and not necessarily connected domains without the effect of climate change, namely when $c=0$. From a different point of view, here we consider  the concentration of the species facing a climate change in an infinite cylindrical domain $\O$ when the exterior domain $\R^N\setminus\O$ becomes extremely unfavorable. To this aim, we first need to ascertain the existence and uniqueness of traveling front for problem (\ref{0}) with Dirichlet boundary condition on  $\partial\O$. The lack of compactness of $\O$ as well as the presence of $c\neq0$ and the fact that near infinity of $\O$, the environment contains both favorable and unfavorable regions   are the main difficulties to be overcome. Here, we will make use of some recent advances of spectral theory in \cite{BR3}.

Finally, the last result is devoted to the study of symmetry breaking of the fronts in $\O.$ The main reason leading to the symmetry breaking is the difference of asymptotic behaviors  near $\pm\infty$. To be more clear, due to the spectral theory for cylindrical domains developed by Berestycki-Nirenberg in \cite{BN}, under some fair assumptions on the growth rate of $f_s(x_1,y,0)$ as $x_1\to\pm\infty$, we can find the exact behaviors of the unique solution of (\ref{0}) as $x_1$ tends to $\pm\infty$ depending on $c$. By conditioning that these behaviors are different, we obtain the asymmetry of the solution. In particular, assuming  $c\neq0$, we will see that the asymmetry holds when $f_s(x_1,y,0)$ converges fast enough to the same negative constant as $x_1\to\pm\infty$.

In the remainder of this section, we give  notations and definitions that are used in the paper. The set $\O$  denotes an infinite straight cylindrical domain $\O=\R\times\omega$, where $\omega$ is an open bounded and smooth domain in $\R^{N-1}$. We  use the notation $x=(x_1,y)\in\R\times\omega$ for the points in $\O$ and denote :
$$\O^+=\{x\in\O, x_1\geq 0, y\in\omega\}\quad;\quad\O^-=\{x\in\O,x_1\leq 0, y\in\omega\};$$
$$\O_r=\{x\in\O,-r< x_1< r, y\in\omega\}.$$
Let  $\mathcal{O}\subset\R^N$ and $L$ be a  uniformly elliptic operator with coefficients bounded on $\mathcal{O}$ defined by
$$L u=a_{ij}(x)\partial_{ij}u(x)+b_i(x)u_i(x)+c(x)u.$$
If $\mathcal{O}$ is smooth and bounded, it is classical that $L$ admits a unique  eigenvalue $-\lambda_D$  (respectively $-\lambda_N$) and a unique (up to multiplication) eigenfunction with Dirichlet (respectively Neumann) boundary condition i.e :
\begin{equation}
\left\{\begin{array}{ll}
 {L}\varphi=-\lambda_D\varphi&x\in\mathcal{O}\\
  \varphi=0 &x\in\partial\mathcal{O}.\nonumber
\end{array} \right.\\\quad\quad\quad\quad
\left\{\begin{array}{ll}
  {L}\varphi=-\lambda_N\varphi&x\in\mathcal{O}\\
  \partial_\nu\varphi=0 &x\in\partial\mathcal{O}.\nonumber
\end{array} \right.
\end{equation}
As is known, the principal eigenpair (eigenvalue and eigenfunction) for an associated elliptic operator plays an important role in deriving persistence results and long time dynamics. In 1994, Berestycki, Nirenberg and Varadhan \cite{BNV} gave a very simple and general definition of the principal eigenvalue of ${L}$ for general domains whose boundaries are not necessarily smooth and later  Berestycki, Hamel and Rossi \cite{BR5} used this approach to define generalized principal eigenvalues in unbounded domains. More precisely now allowing $\mathcal{O}$ to be a smooth and possibly unbounded domain,    they defined the generalized Neumann principal eigenvalue as follow
\begin{eqnarray}
\lambda_N(-{L},\mathcal{O}):=\sup\{\lambda\in\R:\exists\phi\in W^{2,N}_{loc}(\mathcal{O}),\phi>0,({L}+\lambda)\phi\leq0\textrm{ a.e in $\mathcal{O}$, $\partial_\nu\phi\geq0$ on $\partial\mathcal{O}$}\}.\label{2}
\end{eqnarray}
When $\mathcal{O}$ is bounded, these two notions coincide : $\lambda_N(-{L},\mathcal{O})=\lambda_N$. We adopt this definition in our paper. Under the assumption $a_{ij}$, $b_i$, $c\in L^\infty(\mathcal{O})$, it is easily seen that $\lambda_N(-{L},\mathcal{O})$ is well defined. For related definition and more properties of generalized eigenvalues, the reader is referred to \cite{BR3}.

%%%%%%%%%%%%%%%%%%%%%%%%%%%%%%%%%%%%%%%%%%%%%%%%%%%%%%%%%%%%%%%%%%

\subsection{Hypotheses and main results}
\subsubsection{The cylindrical environment without time dependence}
The function $f(x_1,y,s):\R\times\omega\times [0,+\infty)\mapsto\R$ is assumed to be continuous in $x_1$, measurable in $y$, and locally Lipschitz continuous in $s$. In addition, the map $s\mapsto f(x,s)$ is of class $C^{1}(0,s_0)$ for some positive constant $s_0$, uniformly in $x$. We assume that $f(x,0)=0$, $\forall x\in \O$.

As we will see, the dynamics is controlled by traveling fronts. Thus, we look for the solutions of Eq. (\ref{0}) of the type $u(t,x)=U(x_1-ct,y)$, which are called  traveling front solutions with forced speed $c$. Such solutions are given by the equation :
\begin{equation}
\left\{\begin{array}{ll}
  \Delta U+c\partial_1 U+f(x,U)=0 &\textrm{$x\in\O$}\\
  \partial_\nu U=0 & \textrm{$x\in\partial\O$}\\
  \textrm{$U>0$ in $\O$}\\
  \textrm{$U$ is bounded.} \label{3}
\end{array} \right.
\end{equation}
In the results below, we will require the following hypotheses on $f$ :

\begin{equation}
\textrm{$\exists S>0$ such that $f(x,s)\leq 0$ for $s\geq S$,  $\forall x\in \O$,}\label{4}
\end{equation}
\begin{eqnarray}
\begin{array}{cc}
&\textrm{$s\rightarrow f(x,s)/s $ is nonincreasing a.e in $ \O$ and there exist $D\subset\O$, $|D|>0$}\\
&\textrm{such that it is strictly decreasing in $D$}.\label{5}
\end{array}
\end{eqnarray}
Both of these conditions are classical in the context of population dynamics. The first condition means that there is a maximum carrying capacity effect : when the population density is very large, the death rate is higher than the birth rate and the population decreases. The second condition means the intrinsic growth rate decreases when the population density is increasing. This is due to the intraspecific
competition for resources. 

As has been already mentioned, we are looking for a condition that applies to mixed environments. We assume that there exists a measurable bounded function $\mu:\omega\mapsto \R$ such that
\begin{equation}
\mu(y)=\mathop {\limsup}\limits_{|x_1 | \to \infty } f_s(x_1,y,0)\quad\quad\textrm{and}\quad\quad\lambda_\mu:=\lambda_N(-\Delta_y-\mu(y),\omega)>0.\label{6}
\end{equation}
Condition (\ref{6}) means that the mixed environment is globally unfavorable at infinity in the direction of $x_1$. This generalizes  the  condition
\begin{equation}
f_s(x_1,y,0)\leq -m<0 \quad\quad\textrm{for $|x_1|$ large enough, $y\in\omega$},\label{6.1}
\end{equation}
which is used in \cite{BR2}. Indeed, if $\mu(y)\leq-m<0$ one gets  $\lambda_N(-\Delta_y-\mu(y),\omega)\geq m>0$. Our generalization here aims at allowing $f_s(x_1,y,0)$ to change sign  when $|x_1|$ is large.  An illustration of  condition (\ref{6}) will be given in Section 2.1.

We are now ready to state the main results regarding this framework
\subsubsection{The existence and uniqueness  of traveling front}
The existence and uniqueness results are directly conditioned by the amplitude of the speed of climate change and the sign of the principal eigenvalue $\lambda_0:=\lambda_N(-\mathcal{L}_0,\O)$, where $$\mathcal{L}_0\varphi=\Delta\varphi+f_s(x,0)\varphi.$$ 

\subsubsection*{Definition of the critical speed $c^*$}

By using the Liouville transformation $V(x_1,y):=U(x_1,y)e^{\frac{c}{2}x_1}$,  problem (\ref{3}) is equivalent to
\begin{equation}
\left\{\begin{array}{ll}
  \Delta V+f(x_1,y,V(x_1,y)e^{-\frac{c}{2}x_1})e^{\frac{c}{2}x_1}-\dfrac{c^2}{4}V=0 &\textrm{$x\in\O$}\\
  \partial_\nu V=0 & \textrm{$x\in\partial\O$}\\
  \textrm{$V>0$ in $\O$}\\
  \textrm{$V(x_1,y)e^{-\frac{c}{2}x_1}$ is bounded.} \label{7}
\end{array} \right.
\end{equation}
Linearizing this equation about $0$, one gets a self-adjoint operator :
$$\widetilde{\mathcal{L}}w:=\Delta w+(f_s(x,0)-c^2/4)w.$$
We set $\mathcal{L}_0\varphi=\Delta\varphi +f_s(x,0)\varphi$ and  $\lambda_0:=\lambda_{N}(-\mathcal{L}_0,\O)$ is the generalized Neumann principal eigenvalue of $\mathcal{L}_0$ in $\O$. Since $f_s(x,0)$ is bounded, $\lambda_0$ is well defined and finite. We are led to
\begin{definition}
We define the critical speed by
\begin{equation}
c^*:=2\sqrt{-\lambda_0} \quad\quad \textrm{if $\lambda_0<0$}.
\end{equation}
\end{definition}
\begin{proposition}\label{pro:p1}
The eigenvalue $\lambda_N(-\Delta-c\partial_1-f_s(x,0),\O)<0$ iff $0\leq c<c^*.$
\begin{proof}
Let $\mathcal{L}=\Delta+c\partial_1+f_s(x,0)$. Since we do not assume the test-function of (\ref{2}) to be bounded, it immediately follows from the definition (\ref{2}) that  $\lambda_N(-\mathcal{L},\O)=\lambda_N(-\widetilde{\mathcal{L}},\O)=\lambda_0+\frac{c^2}{4}$.
\end{proof}
\end{proposition}

Our first result is

\begin{theorem}\label{thm:T1}
Assume that (\ref{4})-(\ref{6}) hold. Then Eq. (\ref{0})   admits a traveling front solution, that is a solution  of (\ref{3}) if and only if $0\leq c<c^*$. Moreover,  the front is unique when it  exists.
\end{theorem}

This theorem yields an analogue to the results obtained  in \cite{BDNZ},\cite{BR2},\cite{PL}. Indeed,  under the assumption of type (\ref{6}), one should not expect too many times the same thing. We also point out that   the uniqueness of  (\ref{3}) is achieved in the class of positive bounded solutions without necessarily  prescribing the boundary condition as $x_1\to\pm\infty$.

The  next two results deal with the long time dynamics of the evolution equation (\ref{0}) in $L^\infty(\O)$ and $L^1(\O)$. 
\subsubsection{Long time dynamics}

 \begin{theorem}\label{thm:T2}
Let $u(t,x)$ be the solution of (\ref{0})   with initial condition $u(0,x)\in L^\infty(\O)$, which is nonnegative and not identically equal to zero. Assume that   $(\ref{4})-(\ref{6})$ hold.

i) If $c\geq c^*$, then
$$\lim_{t\to\infty}\|u(t,x)\|_{\infty,\O}=0;$$

ii) if $0\leq c<c^*$ then
$$\lim_{t\to\infty}\|(u(t,x_1,y)-U(x_1-ct,y))\|_{\infty,\O}=0,$$
where $U$ is the unique solution of (\ref{3}) and $\|\cdot\|_{\infty,\O}$ denotes the sup-norm on $\O$.
\end{theorem}
This theorem means that  a species cannot keep pace with a climate change if its speed is too large. This theorem generalizes the results in \cite{BDNZ}, \cite{BR2}. Note that, in \cite{BDNZ}, \cite{BR2}, condition (\ref{6.1}) was actually used in the proofs, in particular, to derive the exponential behavior at infinity.  Although our approaches are similar  to those in \cite{BDNZ},  \cite{BR2},  new difficulties arise from the non-constant unfavorable characterization at infinity, especially to obtain the comparison principle. 

 The next result is concerned with the $L^1(\O)$ convergence of the traveling fronts. This result describes the long time dynamics of the total population.

\begin{theorem}
\label{thm:T3}
Let $u(t,x)$ be the solution of $(\ref{0})$ with initial condition $u(0,x)\in L^\infty(\O) \cap L^1(\O)$, which is nonnegative and not identically equal to zero. Assume that  $(\ref{4})-(\ref{6})$ are satisfied then the same conclusions as in Theorem \ref{thm:T2} hold with the  $L^1(\O)$ norm besides $L^\infty(\O)$ norm.
\end{theorem}

%\subsubsection{Examples and illustrations}
%To illustrate the main results obtained in the in the section 1.2.2, we present a further illustration that our theory may be applied. Let $\omega'$ be a subset of $\omega$ with positive measure. We assume that our environment is constantly favorable on domain $\O'=\R\times\omega'\subset\O$, namely there exists a positive constant $f_\infty$ such that $f_s(x_1,y,0)=f_\infty$ in $\O'$. We derive the following result:
%
%\begin{theorem} Let $\omega'$, $\O'$ and $f_\infty$ be previously defined. There exists a threshold value $f^\star$ such that:
%
%i) If $f_\infty<f^\star$, our results (\ref{thm:T1})-(\ref{thm:T3}) apply.
%
%ii) If $f_\infty\geq f^\star$, Eq. (\ref{6.1}) always possess at least a positive solution. 
%\end{theorem}
%
%This theorem yields an interesting  interpretation for the dynamics of population facing a climate change. Although the pace of climate change may be large, if there is a refuge zone favorable enough, the species can always persist by refuging in this zone. Otherwise, if the zone is not sufficiently favorable, the species may go extinct. Furthermore, the threshold to guarantee for certain persistence is determined.

\subsubsection{The partially periodic environment with  time dependence}

We now consider  problem (\ref{0}) in  partially periodic environments with seasonal dependence. Namely, the reaction term $f$ now depends periodically in the time variable and   (\ref{0}) becomes
\begin{equation}
u_t-\Delta u=f({t},x_1-{c}t,y,u) \quad\quad \textrm{$t\in\R$, $x=(x_1,y)\in \R^N $},\label{8}
\end{equation}
where $c>0$ is the given forced speed and $f$ is now assumed to be periodic in $y$. More precisely, we say that the environment is partially periodic in $y$ and depends seasonally on time if :

1) $\forall i\in\{2,...,N\}$, there exist the constants $L_2,...,L_{N}$ such that
\begin{equation}
f(t,x+L_ie_{i},s)=f(t,x,s)\quad\textrm{$\forall t\in\R, s\in\R, x\in\R^N$}\nonumber
\end{equation}
where $\{e_1,...e_N\}$ denotes the unit normal orthogonal basis of $\R^N$.

2) There exists $T>0$, such that 
$$f(t+T,x,s)=f(t,x,s)\quad\textrm{$\forall t\in\R, s\in\R, x\in\R^N$}.$$

We assume in addition that $f(t,x,0)=0$, $f$ is $C^1$ with respect to $s$, $f$ and $f_s$ are Holder-continuous with respect to $t$ and $x$, precisely   
$$\forall s>0,\quad f(\cdot,\cdot,s), f_s(\cdot,\cdot,0)\in C^{\frac{\alpha}{2},\alpha}_{t,x}(\R\times\R^N),$$
where $C^{\frac{\alpha}{2},\alpha}_{t,x}(\mathcal{I}\times\mathcal{H})$, $\mathcal{I}\subset\R$, $\mathcal{H}\subset\R^N$ denotes the space of functions $\phi(t,x)$ such that $\phi(\cdot,x)\in C^{\frac{\alpha}{2}}(\mathcal{I})$ and $\phi(t,\cdot)\in C^{\alpha}(\mathcal{H})$ uniformly with respect to $t$ and $x$  respectively.

We are interested in  pulsating fronts of (\ref{8}), namely  the solutions of the form $u(t,x)=U(t,x_1-ct,y)>0$. They are obtained from the equation
\begin{equation}
\left\{\begin{array}{ll}
  U_t=\Delta U+c\partial_1 U+f(t,x,U)\quad\quad\quad \textrm{$t\in\R,x\in\R^N$}\\
\textrm{$U$ is bounded.}\label{12}
\end{array} \right.
\end{equation}
Note that $U(t,x_1,y)$ is not a priori assumed  to be periodic in $t$ nor in $y$.

Here, we require the  principal eigenvalue of the linearized operator of Eq. (\ref{12}). More generally, we consider the operators of the form :
\begin{equation}
\mathcal{L}u=\partial_tu- a_{ij}(t,x)\partial_{ij}u(t,x)-b_i(t,x)u_i(t,x)-c(t,x)u(t,x),\quad\quad x=(x_1,y)\in\R^N,
\end{equation}
where $a_{ij},b_i,c_i$ are $T$-periodic in $t$ and  periodic in $y$ with the same period. To define the generalized principal eigenvalue of $\mathcal{L}$, we  assume that the coefficients satisfy the regularity condition as mentioned above and the matrix $(a_{ij}(t,x))$ is uniformly elliptic, namely $a_{ij},b_i,c_i\in C^{\frac{\alpha}{2},\alpha}_{t,x}(\R\times\R^N)$ and there exist some positive constants $E_1, E_2$ such that for all $\xi\in\R^{N}$ and $(t,x)\in\R\times\R^N$ such that
$$E_1|\xi|^2\leq \sum_{1\leq i\leq j \leq N}a_{ij}(t,x)\xi_i\xi_j\leq E_2|\xi|^2.$$

\begin{definition}Let $\mathcal{O}\subset\R$ and $\mathcal{Q}=\{(t,x)=(t,x_1,y)\in\R\times\mathcal{O}\times\R^{N-1}\}$, the generalized principal eigenvalue of $\mathcal{L}$ in $\mathcal{Q}$ is defined by:
\begin{equation}
\begin{array}{cc}
\tilde{\lambda}_1(\mathcal{L},\mathcal{Q})=\sup\left\lbrace  \lambda\in\R: \textrm{$\exists$ $\phi>0$, $\phi\in C^{1,2}_{t,x}(\mathcal{Q} )$, $\phi$ is T-periodic in $t$ and  periodic}\right. \\ \left. \textrm{ in $y$ such that $(\mathcal{L}-\lambda)\phi\geq0$ in $\mathcal{Q}$}\right\rbrace. \label{13}
\end{array}
\end{equation}
\end{definition}
By assuming in addition that $a_{ij},b_i,c_i\in L^\infty(\R^{N+1})$, one can take $\lambda=-\sup_{\R\times\R^N}f_s(t,x,0)$ and $1$ as a test function to see that $\tilde{\lambda}_1$ is well-defined and $-\sup_{\R\times\R^N}f_s(t,x,0)\leq \tilde{\lambda}_1$. We point out that this definition does not make sense if we do not require that the test functions to be periodic in $t$. Indeed, since $(\mathcal{L-\lambda})(\phi e^{\alpha t})=(\mathcal{L}+\alpha-\lambda)(\phi e^{\alpha t})$, $\forall \alpha\in\R$, if we do not force the periodicity in $t$, it would yield $\tilde{\lambda}_1=\tilde{\lambda}_1+\alpha$ for all $\alpha$. 
This kind of eigenvalue seems analogous to the ones introduced by Berestycki and Rossi  \cite{BR1} and Nadin  \cite{NA1}. However, the difference is that here we force the test functions to be periodic in $t$ and $y$ but not in $x_1$ while in \cite{BR1}, the test functions are not periodic in any direction of $x=(x_1,y)$ and in \cite{NA1}, the test functions must be periodic in both $t$ and $x=(x_1,y)$.

We further need the two following conditions that are similar to (\ref{4}),(\ref{5}), but    take into account the time-periodic dependence of $f$ :\\
\begin{equation}
\textrm{$\exists S>0$ such that $f(t,x,s)\leq 0$ for $s\geq S$,  $\forall t\in\R, x\in \R^N$,}\label{9}
\end{equation}
\begin{eqnarray}
\begin{array}{cc}
&\textrm{$s\rightarrow f(t,x,s)/s $ is nonincreasing and for all $ t_0\in[0,T)$ there exist $D\subset(-\infty,t_0)\times\R^N,$}\\
&\textrm{$|D|>0$ such that it is strictly decreasing in $D$.}\label{10}
\end{array}
\end{eqnarray}

Suppose that the parabolic operator $\mathcal{\widetilde{L}}$ is defined as follow
$$\mathcal{\widetilde{L}}\phi=\partial_t\phi- a_{ij}(t,y)\partial_{ij}\phi(t,y)-b_i(t,y)\phi_i(t,y)-c(t,y)\phi(t,y),\quad\quad (t,y)\in\R\times\R^{N-1},$$
with  $a_{ij}(t,y),b_i(t,y),c_i(t,y)\in L^\infty(\R\times\R^{N-1})$ and the matrix $a_{ij}(t,y)$ satisfies the uniform elliptic condition. We consider the eigenvalue problem

%  define the generalized space-time periodic eigenvalue of $\mathcal{\widetilde{L}}$ :
%\begin{equation}
%\begin{array}{cc}
%\bar{\lambda}_1(\mathcal{\widetilde{L}},\R\times\R^{N-1})=\sup\{  \lambda\in\R: \textrm{$\exists$ $\phi>0$, $\phi\in C^{1,2}_{t,y}(\R\times\R^{N-1} )$, $\phi$ is T-periodic in $t$} \\  \textrm{and periodic in $y$ such that $(\mathcal{\widetilde{L}}-\lambda)\phi\geq0$ in $\R\times\R^{N-1}$}\}. \nonumber
%\end{array}
%\end{equation}
%It is worth  mentioning that, Nadin \cite{NA1} has shown the existence and uniqueness  of the space-time periodic eigenpair $(\lambda_p,\varphi_p)$ of the eigenvalue problem
\begin{equation}
\left\lbrace\begin{array}{ll}
\widetilde{\mathcal{L}}\varphi=\lambda\varphi\\
\varphi>0\\
\varphi(.,.+T)=\varphi\\
\varphi(.+L_ie_i,.)=\varphi.
\end{array}\right.\label{10.1}
\end{equation}
It is proved by Nadin, Theorems 2.7  \cite{NA1} that there exist a unique eigenpair $(\lambda,\varphi)$, where $\varphi$ is unique up to a multiplicative constant, satisfying (\ref{10.1}). Here, we denote $\lambda(\mathcal{\widetilde{L}},\R\times\R^{N-1})$ by $\lambda$ for sake of simplicity. 
%Obviously, $\lambda_p\leq \bar{\lambda}_1(\mathcal{\widetilde{L}},\R\times\R^{N-1}).$

Using this notion, we assume that there exists a  function $\gamma(t,y)\in L^{\infty}(\R\times\R^{N-1})$, which is periodic in $y$ and T-periodic in $t$ such that 
\begin{equation}
\gamma(t,y)=\mathop {\limsup }\limits_{|x_1|  \to \infty }f_s(t,x_1,y,0)\quad\quad\textrm{and}\quad\quad\lambda=\lambda(\partial_t-\Delta_y-\gamma(t,y),\R\times\R^{N-1})>0.\label{11}
\end{equation}
Condition (\ref{11}) yields the characterization of the environment  expressing that it is globally unfavorable at infinity.  The new difficulties of this problem arise since we deal with the solution in the unbounded domain (whole space) without a-priori assuming  that the solutions are periodic in $y$ nor in $t$. Moreover, the monotonicity in time of solutions of parabolic operators starting by a stationary sub (or super) solution  no longer holds. 

Let us call \begin{equation}\label{12.1}
\mathcal{P}\varphi=\partial_t\varphi-\Delta\varphi-c\partial_1\varphi-f_s(t,x,0)\varphi
\end{equation}
the linearized operator associated with  (\ref{12}). In the sequel, we will briefly denote by $\tilde{\lambda}_1=\tilde{\lambda}_1(\mathcal{P},\R\times\R^{N+1})$. We are now able to state the results of this section.

\begin{theorem}\label{thm:T4}
Assume that $(\ref{9})-(\ref{11})$ hold, then there exists a positive pulsating front $U\in C^{1,2}(\R\times\R^N)$ of  (\ref{12}) if and only if $\tilde{\lambda}_1<0$. If it exists, it is unique, $T$-periodic in $t$, periodic in $y$ and decays exponentially in $|x_1|$, uniformly in $y$ and $t$.
\end{theorem}

One of the interesting points of this theorem is the loss of compactness since $f$ is not periodic in $x_1$ and the solution is not a priori assumed to be periodic in $y$ nor in $t$. We will prove that the uniqueness of (\ref{12}) still holds in this larger class of solution, that is the class of nonnegative bounded solutions. As pointed out in section 1.5 of \cite{NA2}, one cannot expect to show a general uniqueness in the class of nonnegative bounded solutions, even when the coefficients of  (\ref{12}) are periodic in $x=(x_1,y)$ and in $t$  under condition $\tilde{\lambda}_1<0$ only. Some extra assumptions are needed. Here the uniqueness holds due to  assumption (\ref{11}), which is a key ingredient to derive the exponential behavior of solutions of (\ref{12}) as $x_1\to\pm\infty$.

\begin{theorem}\label{thm:T5}
Let $u(t,x)$ be the solution of (\ref{8}) with nonnegative initial datum $u_0(x)\in L^\infty(\R^N)$ and not identically equal to $0$. Assume that (\ref{9})-(\ref{11}) hold.

i) If $\tilde{\lambda}_1\geq0$ then
$$\lim_{t\to\infty} u(t,x)=0,$$
uniformly in $x\in\R^N$.

ii) if  $\tilde{\lambda}_1<0$ then
$$\lim_{t\to\infty}(u(t,x_1,y)-U(t,x_1-ct,y))=0,$$
where $U$ is the unique solution of $(\ref{12})$, uniformly in $x_1$, locally uniformly in $y$. If, in addition, $u_0$ is periodic in $y$ or satisfies
\begin{equation}\label{T5.0}
\forall r>0,\quad\quad\inf_{|x_1|<r,y\in\R^{N-1}}u_0(x_1,y)>0,
\end{equation}
then above convergence is uniform also in $y$. 
\end{theorem}

The fact that the convergence holds uniformly in $x_1$ is a consequence of the exponential decay as $x_1\to\pm\infty$, which can be derived from (\ref{11}).

\textbf{Organization of the paper.} We divide the rest of the paper into four sections.  Section 2 deals with  problem (\ref{0}) in the cylindrical domain $\O$ without time dependence.  Section 3  investigates   problem (\ref{8}) in partially periodic domain with  time dependence. In section 4,  we first study the concentration of the species in $\O$ of Section 2, when the exterior domain becomes extremely unfavorable and further prove the symmetry breaking of the fronts.  Finally, some auxiliary results are contained in the Appendix.

\section{The cylindrical environment without time dependence}

\subsection{An illustration}
Before proving the main results of this section, let us provide an illustration of how theorems  (\ref{thm:T1})-(\ref{thm:T3}) apply and why condition (\ref{6}) is useful to describe the heterogeneity of  habitat of a species facing a climate change.

We consider $\O=\R\times(0,1)$ and for $\alpha\in(0,1),L>0$ the family  $f_{\alpha,L}(x_1,y,s)=(\rho_L(x_1)+\mu_\alpha(y))s-s^2,$ where

\begin{equation}
\rho_L(x_1)=\left\{\begin{array}{ll}
2 & \textrm{on $[-L,L]$}\\
\theta & \textrm{outside $[-L,L]$},
\end{array} \right.\quad\quad \mu_\alpha(y)=\left\{\begin{array}{ll}
1 & \textrm{on $[0,\alpha]$}\\
-1 & \textrm{on $[\alpha,1]$}.
\end{array} \right.\nonumber
\end{equation}
These nonlinearities are discontinuous. However, $\partial_s f_{\alpha,L}(x_1,y,0)$ is well defined a.e and all our results apply for zero order coefficient in $L^\infty$ (see \cite{BR1},\cite{BR2} for further discussion of this extension). We see that, if $\theta\in(-1,1)$, $\partial_s f_{\alpha,L}(x_1,y,0)$ is sign-changing as all the way $|x_1|\to\infty$ and $\mu_\alpha(y)=\lim_{|x_1|\to\infty}\partial_s f_{\alpha,L}(x_1,y,0)$ as $\theta=0$. Now, for every $\alpha\in(0,1)$, let  $(\lambda_\alpha,\phi_\alpha)$ be the (unique) eigenpair of the Neumann eigenvalue problem
\begin{equation}\nonumber
\left\lbrace\begin{array}{ll}
-\phi_\alpha''-\mu_\alpha(y)\phi_\alpha=\lambda_\alpha\phi_\alpha &\textrm{in $(0,1)$}\\
\partial_\nu\phi_\alpha=0&\textrm{at $0$ and $1$.}
\end{array}\right.
\end{equation}
Dividing the equation by $\phi_\alpha$ and integrating by part, we get
$$-\int_0^1\frac{{\phi'}_\alpha^2}{\phi_\alpha^2}dy-\int_0^1\mu_\alpha(y)dy=\lambda_\alpha.$$
Hence, $\lambda_\alpha\leq-\int_0^1\mu_\alpha(y)dy=-(2\alpha-1)$. It is also known that $\lambda_\alpha$ is decreasing with respect to $\alpha$ and that $\alpha\mapsto\lambda_\alpha$ is continuous. Since $\lambda_0=1$ we see that there exists a unique $\ol\alpha$ such that $\lambda_{\ol\alpha}=0$ and $\lambda_\alpha<0$ iff $\alpha>\ol\alpha$.

For $\alpha>\ol\alpha$, $\lambda_\alpha<0$, it is well-known that there exists a unique  positive solution (Berestycki,\cite{B1}) of
\begin{equation}\nonumber
\left\lbrace\begin{array}{ll}
-p_\alpha''-\mu_\alpha(y)p_\alpha+p_\alpha^2=0 &\textrm{in $(0,1)$}\\
\partial_\nu p_\alpha=0&\textrm{at $0$ and $1$.}
\end{array}\right.
\end{equation}
In this case, the environment is globally favorable at infinity and we conjecture that there is  always persistence, namely as $t\to\infty$, $u(t,x_1,y)\to U(x_1-ct,y)$, $\forall(x_1,y)\in\O$, where $U(\xi,y)$ is the unique positive stationary solution of
\begin{equation}\nonumber
\left\lbrace\begin{array}{ll}
u_t=u''+c\partial_1 u+f_{\alpha,L}(\xi,y,u) &\textrm{in $\O$}\\
\partial_\nu u=0&\textrm{on $\partial\O$.}
\end{array}\right.
\end{equation}
If $\liminf_{|x_1|\to\infty}\partial_s f_{\alpha,L}(x_1,y,0)>c^2/4$ uniformly in $y$, this conjecture is true and  we refer to \cite{BR5} for  its proof. However, this is not our current interest.

For $\alpha<\ol\alpha$, the environment is globally unfavorable at infinity. The problem is more subtle and our theory applies in this case. Moreover, if  $\theta<-1$, the environment is completely unfavorable near infinity. For instance, we take $\theta=-2$ and  $c$ not too large, say $c=1$, we claim that there exists a unique threshold value $L^*$ such that the persistence holds iff $L>L^*$.

To prove this, let us denote $\mathcal{Q}_L[\phi]=\phi''+\phi'+\partial_s f_{\alpha,L}(x_1,y,0)\phi$, where $\partial_s f_{\alpha,L}(x_1,y,0)=\rho_L(x_1)+\mu_\alpha(y)$ defined on $\O$. Since $\O$ is unbounded, we cannot define the classical eigenvalue of $\mathcal{Q}_L$ on $\O$. We make use of the definition (\ref{2}). Let us call $\lambda_L=\lambda_N(-\mathcal{Q}_L,\O)$ and $\lambda'_L=\lambda_N(-\phi''-\partial_s f_{\alpha,L}(x_1,y,0)\phi,\O)$. By Proposition \ref{pro:p1}, one has
$$\lambda_L=\lambda_L'+\frac{1}{4}.$$
Since $\rho_L$ is increasing with respect to $L$,  $\lambda_L$ is decreasing with respect to $L$. Moreover, 
the map $L\mapsto\lambda_L$ is continuous on $[0,\infty]$. Indeed, for any $L\in[0,\infty]$, let $\{L_n\}\in[0,\infty]$ be an arbitrary sequence converging to $L$, we see that $\|\partial_s f_{\alpha,L_n}(x_1,y,0)-\partial_s f_{\alpha,L}(x_1,y,0)\|_{L^\infty(\O)}\to0$ as $n\to\infty$. Arguing as in the proof of Proposition 9.2  part (ii) \cite{BR3}, we get $\lim_{n\to\infty}\lambda_{L_n}=\lambda_L.$ It is worth noting that since $\O$ is a smooth domain with Neumann boundary condition, we can apply the  Harnack inequality up to the boundary (see Berestycki-Caffarelli-Nirenberg \cite{BCN}) as in  proof of Proposition 9.2  part (ii) \cite{BR3}. Moreover, since $\partial_s f_{\alpha,0}(x_1,y,0)=-2+\mu_\alpha(y)\leq-1$ and $\partial_s f_{\alpha,\infty}(x_1,y,0)=2+\mu_\alpha(y)\geq1$, by taking $1$ as a test-function, we have $\lambda_0=\lambda_0'+1/4\geq 5/4$ and $\lambda_\infty=\lambda_\infty'+1/4\leq -3/4$. The claim is proved.

As we will see in the next section, for $\alpha<\ol\alpha$, the equation
\begin{equation}\label{V1_ill1}
\left\lbrace\begin{array}{ll}
q_\alpha''+cq'+f_{\alpha,L}(x_1,y,q)=0 &\textrm{in $\O$}\\
\partial_\nu q_\alpha=0&\textrm{on $\partial\O$}
\end{array}\right.
\end{equation}
admits a unique positive solution if and only if $\lambda_L<0$. From this result, we can also prove that $u(t,x_1,y)$   converges as $t\to\infty$  to the unique positive solution $q(x_1,y)$ of (\ref{V1_ill1}) if $\lambda_L<0$ and $u(t,x_1,y)$ converges to zero if $\lambda_L\geq0$. Even if the persistence  is known, i.e $\lambda_L<0$, the  non-persistence and the uniqueness are still  delicate questions.

\subsection{The existence and uniqueness of the front}
To achieve the existence and uniqueness of Eq. (\ref{3}), the key property is the exponential decay of solution. This estimate is the object of the next section

\subsubsection{Exponential decay}

\begin{proposition}\label{lem0} Let $U$ be a nonnegative bounded solution of (\ref{3}). Assume that (\ref{5})-(\ref{6}) hold, then for all $0<\alpha<\alpha^*$ with $\alpha^*=\frac{-c+\sqrt{c^2+4\lambda_\mu}}{2}$, there exists a positive constant $C(\alpha)$  such that
$$U(x_1,y)+|\nabla U(x_1,y)|\leq C(\alpha)e^{-\alpha|x_1|}.$$ 
 \end{proposition}
 
We point out that this proposition generalizes  Proposition 3 of \cite{BR2}. Moreover, our proof is quite simpler than the proof of Berestycki and Rossi in  that we do not use the Liouville transformation.

\begin{proof} 
Consider $\varphi\in W^{2,p}(\omega)$ the eigenfunction associated with $\lambda_\mu$  
\begin{equation}
\left\{\begin{array}{ll}
\Delta \varphi+\mu(y)\varphi+\lambda_\mu\varphi=0&\textrm{in $\omega$}\\
 \partial_\nu\varphi(y)=0 &\textrm{on $\partial\omega$}\label{T1.0.1}
\end{array}\right.
\end{equation}
where $\nu=\nu(y)$ is the outward unit normal on $\omega$. As is well-known by the Hopf lemma, $\inf_{\ol\omega}{\varphi}>0$. 

For any $\delta\in(0,\lambda_\mu)$, let  $\mathcal{L}w=\Delta w+c\partial_1w+(\mu(y)+\delta) w$ be defined in $\O$. Due to (\ref{5})-(\ref{6}), one has
$$\Delta U+c\partial_1 U+f_s(x,0)U\geq0 \quad\quad\textrm{in $\O$}.$$
and there exists $R=R(\delta)>0$ such that
$$f_s(x_1,y,0)\leq \mu(y)+\delta\quad\textrm{in $\O\setminus\O_R$},$$
therefore $\mathcal{L}U\geq0$ in $\O\setminus\O_R$. For any $p>0$, set
$$w_p(x_1,y)=e^{(R+p)(\tau-\alpha)}e^{\alpha|x_1|}\varphi(y)+e^{R(\tau+\alpha)}e^{-\alpha|x_1|}\varphi(y)=C_1e^{\alpha|x_1|}\varphi(y)+C_2e^{-\alpha|x_1|}\varphi(y),$$
where $R,\tau,\alpha>0$ will be chosen.  Direct computation shows that
\begin{eqnarray}
\frac{\mathcal{L}w_p}{\varphi}&=&C_1\left(\alpha^2+\frac{\Delta_y\varphi}{\varphi}+c\alpha\frac{x_1}{|x_1|}+\mu(y)+\delta\right) e^{\alpha|x_1|}+\nonumber\\& &\quad\quad\quad\quad\quad\quad\quad\quad\quad\quad\quad\quad\quad\quad\quad+C_2\left(\alpha^2+\frac{\Delta_y\varphi}{\varphi}-c\alpha\frac{x_1}{|x_1|}+\mu(y)+\delta\right) e^{-\alpha|x_1|}\nonumber\\
&\leq&(\alpha^2+c\alpha-\lambda_\mu+\delta)(C_1e^{\alpha|x_1|}+C_2e^{-\alpha|x_1|})\quad\quad\quad\textrm{in $\O\setminus\O_R$}\label{lem0.1}
\end{eqnarray}
For $w_p$ to be a supersolution of $\mathcal{L}$ in $\O\setminus\O_R$, it suffices to take $$\alpha=\alpha(\delta)=\dfrac{-c+\sqrt{c^2+4\lambda_\mu-4\delta}}{2}>0.$$
Clearly, $\alpha(\delta)$ is decreasing with respect to $\delta\in(0,\lambda_\mu)$. Choosing $\tau=\alpha/2$ and  $R$ large enough , we have the following estimates on $\partial (\O_{R+p}\setminus\O_R)$
$$\left\lbrace\begin{array}{ll}
w_p(x_1,y)\geq e^{R\tau}\varphi(y)\geq e^{R\tau}\inf_\omega\varphi\geq U(x_1,y)&\textrm{as $|x_1|=R,y\in\omega$}\\
w_p(x_1,y)\geq e^{(R+p)\tau}\varphi(y)\geq e^{(R+p)\tau}\inf_\omega\varphi\geq U(x_1,y)&\textrm{as $|x_1|=R+p,y\in\omega$.}
\end{array}\right.
$$
Fix $\alpha$, $\tau$ and $R$, we set $z(x_1,y)=\frac{w_p(x_1,y)-U(x_1,y)}{\varphi(y)}$. Obviously, $\mathcal{L} (z\varphi)=\mathcal{L} w_p-\mathcal{L} U\leq0$. Routine computation yields
\begin{eqnarray}
\mathcal{L}_1[z]=\frac{\mathcal{L}[z\varphi]}{\varphi}&=&\Delta z+2\dfrac{\nabla_y\varphi}{\varphi}\cdot\nabla_y z+c\partial_1z+\frac{\Delta\varphi+(\mu(y)+\delta)\varphi}{\varphi}z \leq0\nonumber
\end{eqnarray}
Observe that $z\geq 0$ when $|x_1|\in\{R,R+p\}$, $y\in\omega$ and $\partial_\nu z=0$ on $\partial\omega$. Moreover, the zero order coefficient of  $\mathcal{L}_1$ is negative. Hence, by the maximum principle, we get 
$$U(x_1,y)\leq w_p(x_1,y)=e^{-(R+p)\alpha/2}e^{\alpha|x_1|}\varphi(y)+e^{3R\alpha/2}e^{-\alpha|x_1|}\varphi(y)\quad\quad\textrm{in $\O_{R+p}\setminus \O_R$}.$$
Letting $p\to\infty$, we obtain
\begin{equation}
U(x_1,y)\leq e^{3R\frac{-c+\sqrt{c^2+4\lambda_\mu-4\delta}}{4}}e^{-\frac{-c+\sqrt{c^2+4\lambda_\mu-4\delta}}{2}|x_1|}\varphi(y)\quad\quad\textrm{in $\O\setminus\O_R$}.\nonumber
\end{equation}
Set  $\alpha^*=\frac{-c+\sqrt{c^2+4\lambda_\mu}}{2}$, one sees that when $\delta$ goes to $0$, $\alpha$ is arbitrarily close to $\alpha^*$.  Since $\varphi$ is bounded, obviously for all $0<\alpha<\alpha^*$ we can choose  $C(\alpha)$ (possibly changed if necessary) such that 
$$U(x_1,y)\leq C(\alpha)e^{-\alpha|x_1|}\quad\quad\textrm{in $\O$}.$$
This implies that $U(x_1,y)$ decays exponentially  as $|x_1|\to\infty$, uniformly in $y$. On the other hand, since $U$ is a solution of  (\ref{3}), we can use $L^p$ estimate for the Neumann problem with $p>N$ and Harnack inequality up to the boundary (see Berestycki-Caffarelli-Nirenberg \cite{BCN}) to derive
\begin{equation}
\|\nabla U\|_{L^\infty( B_1(x))}\leq C_1\|U\|_{W^{2,p}(B_1(x))}\leq C_2\|U\|_{L^\infty( B_2(x))}\leq C_3 U(x).\label{lem0.2}
\end{equation}
These inequalities end the proof.
\end{proof}
\begin{remark} \label{Remark1} Note that if $U$ is only a subsolution of (\ref{3}), we do not have the estimates in (\ref{lem0.2}) for  $|\nabla U|$. However, in the proof of Theorem \ref{thm:T1} below, we only require the decay of $U$ as $|x|\to\infty$, which is also true if $U$ is a subsolution of (\ref{3}).
\end{remark}

\subsubsection{Proof of Theorem \ref{thm:T1}}
Let us first  consider the case  $0\leq c< c^*$. By Proposition \ref{pro:p1}, we know that $$\tilde{\lambda}_1:=\lambda_N(-\Delta-c\partial_1-f_s(x_1,y,0),\O)<0.$$ Thanks to Proposition 1, \cite{BR2}, we have the limit $\mathop {\lim }\limits_{R  \to \infty }\lambda_R=\tilde{\lambda}_1<0$, where $\lambda_R$ is the unique eigenvalue of problem :
\begin{eqnarray}
\left\{\begin{array}{ll}
-\Delta\varphi_R-c\partial_1\varphi_R-f_s(x,0)\varphi_R=\lambda_R\varphi_R& x\in\O_R\\
\varphi_R(x)>0& x\in\O_R \\
\partial_\nu\varphi_R(x_1,y)=0 & |x_1|<R, y\in\partial\omega\\
\varphi_R(\pm R,y)=0& y\in\omega.
\end{array} \right.\nonumber
\end{eqnarray}
Moreover there exists an eigenfunction $\varphi_\infty\in W^{2,N}(\O)$ associated with $\tilde{\lambda}_1$. Fix $R>0$ large enough such that $\lambda_R<0$, we define $\phi(x)$ as following :
$$\phi(x)=\left\{\begin{array}{ll}
\varphi_R(x)& x\in\O_R\\
0&\textrm{otherwise}.
\end{array} \right.
$$
Since $f(x,s)$ is of class $C^1[0,s_0]$ with respect to $s$, for $\varepsilon>0$ small enough, we see that 
$$\Delta(\varepsilon\phi)+c\partial_1(\varepsilon\phi)+f(x_1,y,\varepsilon\phi)=\varepsilon\phi\left[-\lambda_R+\frac{f(x_1,y,\varepsilon\phi)}{\varepsilon\phi}-f_s(x_1,y,0)\right]>0.$$
Hence, $\varepsilon\phi$ is a subsolution of Eq. (\ref{3}). Since $\phi$ is compactly supported, we can choose $\varepsilon$  small such  that $\varepsilon\sup\phi\leq S$, where $S$ is a super solution of Eq. (\ref{3}) given by (\ref{4}). Therefore, by the classical iteration method, there exists a nonnegative solution $U$ satisfying $\varepsilon\phi\leq U\leq S$. Furthermore, thanks to the strong maximum principle, $U$ is  strictly positive .

The nonexistence and uniqueness are  direct consequences of the following comparison principle. Let  $U$ and $V$ be respectively a super-and sub-solution of (\ref{3}). We will now show  that $V(x)\leq U(x)$ in $\O$. Indeed, by  condition (\ref{6}), there exists an eigenpair $(\lambda_\mu,\varphi)$ of (\ref{T1.0.1}), where $\varphi$ satisfies $\inf_{\ol\omega}\varphi>0$ due to
the Hopf lemma. On the other hand, Proposition \ref{lem0} and Remark (\ref{Remark1}) imply that  $V$ decays exponentially  as $|x_1|\to\infty$, uniformly in $y$, therefore, for any $\varepsilon>0$, there exist $R(\varepsilon)>0$ such that $V(x_1,y)\leq \varepsilon\varphi(y)$ in $\O\setminus \O_{R(\varepsilon)}$. Then, the set
$$K_\varepsilon:=\{k>0:k U\geq V-\varepsilon\varphi \textrm{ in } \overline{\O}\},$$
 is nonempty. Let us call $k(\varepsilon):=\inf K_\varepsilon$. Obviously, the function $k(\varepsilon):\mathbb{R}^+\to\R$ is nonincreasing. Assume by a contradiction $$k^*=\mathop{\lim}\limits_{\varepsilon\to0^+}k(\varepsilon)>1.$$
Take $0<\varepsilon<\sup_\O V/\sup_\omega\varphi$, we have $k(\varepsilon)>0$, $k(\varepsilon)U-V+\varepsilon\varphi\geq0$. By the definition of $k(\varepsilon)$, there exists a sequence $(x^\varepsilon_{1,n},y^\varepsilon_{n})$ in $\overline{\O}$ such that
$$\left(k(\varepsilon)-\dfrac{1}{n}\right)U(x^\varepsilon_{1,n},y^\varepsilon_{n})<V(x^\varepsilon_{1,n},y^\varepsilon_{n})-\varepsilon\varphi(y^\varepsilon_{n}).$$
Fix $\varepsilon>0$, we have $(x^\varepsilon_{1,n},y^\varepsilon_n)\in\O_{R(\varepsilon)}$ for $n$ large enough, therefore $(x^\varepsilon_{1,n},y^\varepsilon_n)$  converges up to subsequence to some $(x_1(\varepsilon),y(\varepsilon))\in\O_{R(\varepsilon)}$. This limiting point must satisfy :
\begin{equation}
(k(\varepsilon)U-V+\varepsilon\varphi)(x_1(\varepsilon),y(\varepsilon))=0.\label{T1.1}
\end{equation}

Without loss of generality, we may assume $\mathop {\lim }\limits_{\varepsilon  \to 0^+}y(\varepsilon)=y_0\in\ol\omega$. The case that there exists $x_0$ such that $|x_0|=\mathop {\lim\inf }\limits_{\varepsilon  \to 0^+}|x_1(\varepsilon)|<\infty$ is ruled out. Indeed, from (\ref{T1.1}), $k^*<\infty$, the function $W=k^*U-V$ is nonnegative and vanishes at $(x_0,y_0)$. Since $f$ is Lipschitz continuous with respect to second variable and $k^*>1$, we have
$$-\Delta W-c\partial_1 W\geq k^*f(x,{U})-f(x,V)\geq f(x,k^*{U})-f(x,V)\geq z(x)W,$$
for some function $z(x)\in L^\infty_{loc}(\O)$. Thanks to  condition (\ref{5}), this inequality holds strictly in $D\subset\O$, with $|D|>0$. The strong maximum principle implies that $W$ cannot achieve a minimum value in the interior of $\O$. This means $y_0\in\partial\omega$, but the Hopf lemma yields another contradiction: $\partial_\nu W(x_0,y_0)<0$.

It remains to consider the case $\mathop {\lim}\limits_{\varepsilon  \to 0^+}|x_1(\varepsilon)|=\infty$. Set $W^\varepsilon=k(\varepsilon){U}-V+\varepsilon\varphi$, we have $W^\varepsilon\geq0$ and $W^\varepsilon$ vanishes at $(x_1(\varepsilon),y(\varepsilon))$. Thus there  exists  $r>0$ such that  $k(\varepsilon){U}<V$ in $B_r(x_1(\varepsilon),y(\varepsilon))\cap\O$. For $\varepsilon$ small enough, $k(\varepsilon)>1$, we derive from (\ref{5}) for $B_r(x_1(\varepsilon),y(\varepsilon))\cap\O$:
\begin{eqnarray}
(\Delta +c\partial_1) W^\varepsilon&\leq& f(x,V)-k(\varepsilon)f(x,U)-(\mu(y)+\lambda_\mu)\varepsilon\varphi\leq f(x,V)- f(x,k(\varepsilon){U)}-(\mu(y)+\lambda_\mu)\varepsilon\varphi\nonumber\\
&\leq&-\dfrac{f(x,k(\varepsilon){U})}{k(\varepsilon){U}}(k(\varepsilon){U}-V+\varepsilon\varphi)-\dfrac{\lambda_\mu}{2}\varepsilon\varphi-\left(\dfrac{\lambda_\mu}{2}+\mu(y)-\dfrac{f(x_1,y,k(\varepsilon){U})}{k(\varepsilon){U}}\right)\varepsilon\varphi.\nonumber\\ \label{T1.2}
\end{eqnarray}
Take $0<\varepsilon \ll 1$, then $|x_1(\varepsilon)|\gg1$, we have 
$$\dfrac{f(x_1,y,k(\varepsilon){U})}{k(\varepsilon){U}}<\mu(y)+\dfrac{\lambda_\mu}{2},\quad\quad\forall(x_1,y)\in B_r(x_1(\varepsilon),y(\varepsilon))\cap\O,$$
choosing $r$  smaller if necessary. Since $\lambda>0$, we get from (\ref{T1.2})
$$-\Delta W^\varepsilon-c\partial_1 W^\varepsilon-\varrho(x)W^\varepsilon> \dfrac{\lambda_\mu}{2}\varepsilon\varphi>0\quad\textrm{in $B_r(x_1(\varepsilon),y(\varepsilon))\cap\O$},$$
where $\varrho(x)=\frac{f(x,k(\varepsilon){U})}{k(\varepsilon){U}}$ is bounded. The strong maximum principle asserts  that $(x_1(\varepsilon),y(\varepsilon))$ cannot be an interior point of $\O$. Hence $(x_1(\varepsilon),y(\varepsilon))\in\partial\O$, but then the Hopf lemma  yields another contradiction $\partial_\nu W^\varepsilon(x_1(\varepsilon),y(\varepsilon))<0$.

We have proved that $k^*=\mathop{\lim}\limits_{\varepsilon\to0^+}k(\varepsilon)\leq 1$. Letting $\varepsilon\to0^+$, we derive
$$V\leq\mathop{\lim}\limits_{\varepsilon\to0^+}(k(\varepsilon)U+\varepsilon\varphi)\leq U \quad\quad\textrm{in $\O$}.$$
The uniqueness of Eq. (\ref{3}) is obviously achieved by exchanging the roles of $U$ and $V$. We end the proof by showing the nonexistence  when $c\geq c^*$. Assume by contradiction that  (\ref{3}) possesses a positive solution $U$ when  $c\geq c^*$. One has $\tilde{\lambda}_1=\lambda_N(-\Delta-c\partial_1-f_s(x_1,y,0),\O)\geq0$. Let $\varphi_\infty$ be a generalized principal eigenfunction with Neumann boundary condition associated with $\tilde{\lambda}_1$. Without loss of generality, we may assume that $0<\varphi_\infty(0)<U(0)$. We derive, from (\ref{5}), that
$$
-\Delta\varphi_\infty-c\partial_1\varphi_\infty=(f_s(x_1,y,0)+\tilde{\lambda}_1)\varphi_\infty\geq f(x_1,y,\varphi_\infty)
\quad\quad\textrm{in $\O$}
$$
By Proposition \ref{lem0},  $U$ decays exponentially as $|x_1|\to\infty$ uniformly in $y$. The above comparison principle implies that $U(x)\leq \varphi_\infty(x)$ for all $x\in\O$. This contradiction ends the proof.

\begin{remark} The globally unfavorable characterization of the environment near infinity plays the key role  in proving the uniqueness of (\ref{3}). Indeed,  if $f$ is homogeneous in the traveling direction, i.e  independent of $x_1$ and (\ref{3}) admits a positive solution $U(x_1,y)$, then for all $a\in\R$, $U(x_1+a,y)$ are also solutions of (\ref{3}) and thus the  uniqueness does not hold.
\end{remark}

\subsection{Long time dynamics}

In order to study the long time dynamics of Eq. (\ref{0}), we first prove the  following Liouville type theorem for entire solutions (solutions for all $t\in\R$). Consider the evolution problem in the cylindrical domain with Neumann boundary condition 
\begin{equation}
\left\{\begin{array}{ll}
\partial_tu^*=\Delta u^*+c\partial_1 u^*+f(x,u^*) &\textrm{$t\in\R$, $x\in\O$}\\
\partial_\nu u^*=0 &\textrm{$t\in\R$, $x\in\partial\O$,}\label{V1T2.1}
\end{array} \right.
\end{equation}
we have the following auxiliary result 
\begin{theorem}\label{thm:T'2}
Assume that $u^*$ is a nonnegative bounded  solution to (\ref{V1T2.1}) and conditions $(\ref{4})-(\ref{6})$ are satisfied, then   $u^*\equiv0$ if $c\geq c^*$, where $c^*$ is defined in Proposition \ref{pro:p1}. Conversely, if $c<c^*$ and there exist a sequence $(t_n)\in\R$ as $n\to\infty$ and a point $x_0\in\overline{\O}$ such that
\begin{equation}
\lim_{n\to\infty}t_n=+\infty,\quad\quad\quad\quad\liminf_{n\to\infty}u^*(-t_n,x_0)>0,\label{V1T2.2}
\end{equation}
then $u^*(t,x)\equiv U(x),$ where $U(x)$ is the unique solution of $(\ref{3}),$ given by Theorem \ref{thm:T1}.
\end{theorem}

\begin{proof}  Set $S^*=\max\{S,\|u^*\|_{L^\infty(\O)}\}$, where $S$ is the positive constant given in (\ref{4}), obviously $S^*$  is a super solution of stationary equation of Eq. (\ref{V1T2.1}). Let $v(t,x)$ be the solution of (\ref{V1T2.1}) starting by $v(0,x)=S^*$,  the maximum principle infers that $v$ is nonincreasing in $t$. Using standard parabolic estimates up to the boundary \cite{LI} and compact injection theorems, we see that $v$ converges locally uniformly in $\O$ to a stationary solution $V(x)$ of  (\ref{V1T2.1}). That $V(x)$ solves Eq. (\ref{3}). For any $h\in\R$, we define $v_h(t,x)=v(t-h,x)$. This function is a solution of (\ref{V1T2.1}) in $(h,+\infty)\times\O$ and satisfies $v_h(h,x)=S^*\geq u^*(h,x)$. The  parabolic comparison principle yields 
\begin{equation}
0\leq u^*(t,x)\leq\lim_{h\to-\infty}v_h(t,x)=V(x)\quad\quad\forall t\in\R,x\in\O.\label{V1T2.2.1}
\end{equation}
We consider separately two different cases.

Case 1. $c\geq c^*$. 

Theorem \ref{thm:T1} asserts that the stationary equation of Eq. (\ref{V1T2.1}) only has zero-solution. Namely, $V(x)\equiv0$ in $\O$. Therefore, the necessary condition for existence of nontrivial entire solution of (\ref{V1T2.1}) is $c< c^*$.

Case 2. $c< c^*$ and (\ref{V1T2.2}) holds.

 Theorem \ref{thm:T1} again asserts that the stationary  equation of Eq. (\ref{V1T2.1}) admits a unique positive solution $U$. We will prove that $u^*(t,x)\equiv U(x)$. Assume by contradiction that there exists  $x_0\in\O$ such that $u^*(t,x_0)\neq U(x_0)$. We will reach a contradiction by proving the following claim.\\

Claim. There exist $\varepsilon\in(0,1]$ and $n_0\in\N$ such that for $n\geq n_0$, one has  $\varepsilon U(x)\leq u^*(-t_n,x)$. \\

Assume for a moment that this claim holds true, the concluding argumentation goes as follows. Thanks to (\ref{5}), for any $\varepsilon\in(0,1]$, $\varepsilon U$ is a subsolution of stationary equation  of Eq. (\ref{V1T2.1}). Let  $w(t,x)$ be a solution of (\ref{V1T2.1}) with initial condition $w(0,x)=\varepsilon U(x)$ and $w_n(t,x)=w(t+t_n,x)$. We know, by the standard parabolic estimates, that as $t\to\infty$, $w_n(t,x)$ is nondecreasing, bounded from above by $S^*$ and converges locally uniformly in $\O$ to the unique stationary solution $W(x)$ of Eq. (\ref{V1T2.1}). The strict positivity of $W$ is derived from the condition $c<c^*$. By the way of setting, one has $w_n(-t_n,x)=\varepsilon U(x)\leq u^*(-t_n,x)$. The parabolic comparison principle implies that $w_n(t,x)\leq u^*(t,x)$ in $(-t_n,+\infty)\times\O$. Therefore, by letting $n\to\infty$, one has
$$u^*(t,x)\geq \lim_{n\to\infty} w_n(t,x)= W(x)\quad\quad\textrm{locally in $\R\times\O$}.$$
Combining this inequality with (\ref{V1T2.2.1}), we obtain $W(x)\leq u^*(t,x)\leq V(x)$, $\forall t\in\R$, $\forall x\in \O$. The uniqueness result of Theorem \ref{thm:T1} yields $u^*\equiv W\equiv V$.

It remains to prove the claim. Assume by contradiction that for all  $ \varepsilon\in(0,1]$ and for all $n_0\in\N$ there exist $n(\varepsilon)>n_0$  and  $x_{n(\varepsilon)}\in\O$ so that  $\varepsilon U(x(n_\varepsilon))> u^*(-t_{n(\varepsilon)},x_{n(\varepsilon)})$. Since $U$ is bounded, choosing a sequence $\varepsilon_k\to0$ as $k\to\infty$, by a diagonal extraction, one finds sequences $(t_k)\in\R^+$ and $(x_k)\in\O$ such that $t_k\to+\infty$ and $u^*(-t_k,x_k)\to0$ as $k\to\infty$.  We set $$\tilde{u}_k(t,x)=u^*(t+t_k,x+x_k).$$
Obviously, $\tilde{u}_k(t,x)$ is bounded from above by $S^*$ and satisfies the equation
\begin{equation}
\left\{\begin{array}{ll}
\partial_t\tilde{u}_k=\Delta \tilde{u}_k+c\partial_1 \tilde{u}_k+f(x+x_k,\tilde{u}_k) &\textrm{$t\in\R$, $x\in\O$,}\\
\partial_\nu \tilde{u}_k=0 &\textrm{$t\in\R$, $x\in\partial\O$}\nonumber
\end{array} \right.
\end{equation}
By  standard parabolic estimates and and compact injection theorems, as $k\to\infty$, we get $\tilde{u}_k\to\tilde{u}_\infty$ (up to subsequences)  locally uniformly in $\R\times\O$. Thanks to the Lipschitz continuity of $f(x,s)$ with respect to $s$, there exists a negative constant $-M$ so that $\tilde{u}_\infty$ satisfies the equation
\begin{equation}
\left\{\begin{array}{ll}
\partial_t\tilde{u}_\infty\geq \Delta \tilde{u}_\infty+c\partial_1\tilde{u}_\infty-M\tilde{u}_\infty &\textrm{$t\in\R$, $x\in\O$}\\
\partial_\nu\tilde{u}_\infty=0&\textrm{$t\in\R$, $x\in\partial\O$.}\nonumber
\end{array} \right.
\end{equation}
Moreover,  $\tilde{u}_\infty(0,0,0)=0$. The strong maximum principle implies that $\tilde{u}_\infty(t,x)=0$, $\forall t\leq0,x\in\O$. Choosing $t=-2t_k$, we get $\lim_{k\to\infty}{u}^*(-t_k,x)=0$, $\forall x\in\O$. This contradicts   assumption (\ref{V1T2.2}) and thus   completes the proof.
\end{proof}
\begin{remark} In this proof, we have used different arguments from the ones of Berestycki and Rossi, Lemma 3.4 \cite{BR2}. More precisely, we choose a  solution of Eq. (\ref{V1T2.1}), $w(t,x)$ starting by a subsolution of stationary equation $\varepsilon U(x)$, which is not necessarily compactly supported but bounded. On the other hand, we reach the contradiction by showing that $\lim_{k\to\infty}{u}^*(-t_k,x)=0$, $\forall x\in\O$, which differs from the way to show that for all $r>0$,  $\liminf_{n\to\infty,x\in\O_r}u^*(t_n,x)>0$ in \cite{BR2}.

\end{remark}

We are now ready to prove Theorem \ref{thm:T2} to derive long time behavior of solution of (\ref{0}) in $L^\infty(\O)$.

\begin{proof}[\textbf{Proof of \thm{T2}}]
Let $S':=\max\{S,\|u_0\|_{L^\infty(\O)}\}$, where $S$ is the positive constant in (\ref{4}). Then $0$ and $S'$ are respectively sub and super solution of (\ref{0}). It follows from \cite{LI}, by the standard theory of semilinear parabolic equations, that there exists a unique (weak) solution to (\ref{0}) satisfying $0\leq u\leq S'$ with initial condition $u_0(x)$. We deduce, from the parabolic strong maximum principle, that $u(t,x)>0$ $\forall t>0, x\in\overline{\O}$ (by extending $u(t,x)$ to larger cylinder to make the "corner" smooth). The locally long time behavior of $u$ follows by applying directly Theorem \ref{thm:T'2} and the standard parabolic estimates. Actually, one sets $\tilde{u}(t,x_1,y)=u(t,x_1+ct,y)$. The solution of this type satisfies $\tilde{u}(0,x)=u_0(x)$ and  satisfies
\begin{equation}
\left\{\begin{array}{ll}
\partial_t\tilde{u} =\Delta \tilde{u} +c\partial_1 \tilde{u} +f(x,\tilde{u} ) &\textrm{$t>0$, $x\in\O$}\\
\partial_\nu \tilde{u} =0 &\textrm{$t>0$, $x\in\partial\O$}.\label{V1T2.3}
\end{array} \right.
\end{equation}
To apply Theorem \ref{thm:T'2}, we only need to verify condition (\ref{V1T2.2}) when $c<c^*$. Indeed, the first case $c\geq c^*$ is easily seen. Let $(t_n)$ be a sequence such that $t_n\to+\infty$ as $n\to\infty$, we infer, by the parabolic estimates and embedding theorems, that the sequence $\tilde{u}(t+t_n,x)$ converges (up to subsequences) to some nonnegative bounded solution $u^*(t,x)$ of Eq. (\ref{V1T2.1}) as $n\to\infty$ locally in $\O$. By Theorem \ref{thm:T'2}, this limit is identically equal to $0$ when $c\geq c^*$. Consider the case $c<c^*$, we necessarily verify  condition (\ref{V1T2.2}). Let $U$ be the unique solution of stationary solution of (\ref{V1T2.3}) and $(t_n)$ be such that $t_n\to-\infty$ as $n\to\infty$. Fix $R>0$, the Hopf lemma implies that $\inf_{\overline{\O}_R}\tilde{u}(1,x)>0$. For $\varepsilon>0$, the function $\varepsilon U$ is a subsolution to stationary equation of (\ref{V1T2.3}) when $\varepsilon\leq 1$. Take $\varepsilon$ small enough such that  $\varepsilon U\leq\tilde{u}(1,x)$ in $\overline{\O}_R$. Hence $(t,x)\mapsto\varepsilon U(x)$ is a subsolution to (\ref{V1T2.3}) in $\R\times\O_R$. The parabolic comparison yields $\varepsilon{U}(x)\leq \tilde{u}(t+1,x)$ for $t>0$ and $x\in\O_R$. As a consequence 
$$\inf_{t\in\R}u^*(t,0,y_0)\geq U(0,y_0)>0,\quad\quad\textrm{for some $y_0\in\omega$}.$$

It remains to show that  the convergences hold uniformly in $\O$. Assume by  contradiction that $$\lim_{t\to\infty}\tilde{u}(t,x)=U(x)$$ is not uniform in $x\in\O$. This means that there exist  $\varepsilon>0$, $(t_n)\in\R^+$ and $(x_{1,n}, y_n)\in\O$ such that
$$\lim_{n\to\infty}t_n=\infty,\quad\quad|\tilde{u}(t_n,x_{1,n},y_n)-U(x_{1,n},y_n)|\geq \varepsilon\quad\forall n\in\N.$$
Since $y_n\in\omega$, which is bounded, one may assume that $y_n$ converges (up to subsequences) to $\zeta\in\overline{\omega}$. The locally uniform convergences  yields  $\lim_{n\to\infty}|x_{1,n}|=\infty$, therefore $\lim_{n\to\infty}U(x_{1,n},y_n)=0$ in both cases $c\geq c^*$ and $c<c^*$. Then we get
$$\liminf_{n\to\infty}\tilde{u}(t_n,x_{1,n},y_n)\geq \varepsilon.$$
The standard parabolic estimates and compact injections again imply that $\tilde{u}(t+t_n,x_1+x_{1,n},y_n)$  converges (up to subsequences) to $\tilde{u}_\infty(t,x_1,\zeta)$ uniformly in $(-\rho,\rho)\times\O_\rho$, for any $\rho>0$. In particular, $\tilde{u}_\infty$ satisfies $\tilde{u}_\infty(0,0,\zeta)\geq\varepsilon$ and satisfies the following equation
\begin{equation}
\left\{\begin{array}{ll}
\partial_t\tilde{u}_\infty\leq \Delta \tilde{u}_\infty+c\partial_1\tilde{u}_\infty+\mu(y)\tilde{u}_\infty &\textrm{$t\in\R$, $x\in\O$}\\
\partial_\nu\tilde{u}_\infty=0&\textrm{$t\in\R$, $x\in\partial\O$.}
\end{array} \right.
\end{equation}
By  condition (\ref{6}),  there exists an eigenpair $(\lambda_\mu,\varphi)$ of Eq. (\ref{T1.0.1}) satisfying $\lambda_\mu>0$. Setting $\upsilon(t,y)=S''e^{-\lambda_\mu (t+h)}\varphi(y)$, we have
$$\partial_t\upsilon-\Delta\upsilon-c\partial_1\upsilon-\mu(y)\upsilon=0.$$
We know, by the Hopf lemma, that $\inf_{\omega}\varphi(y)>0$. Hence, the function $W(t,x)=\tilde{u}_\infty(t,x)-\upsilon(t,y)$ satisfies $W(-h,x)\leq 0$ for $S''$ large enough. Let us call $\mathcal{L}_1=\partial_t-\Delta-c\partial_1-\mu(y)$, then
$$\left\{\begin{array}{ll}
\mathcal{L}_1W\leq0 &\textrm{$t\in\R$, $x\in\O$}\\
\partial_\nu W\leq0&\textrm{$t\in\R$, $x\in\partial\O$.}
\end{array} \right.
$$
Set $W(t,x)=z(t,x)\varphi(y)$ we have $\partial_\nu z\leq0$ on $\partial\O$, $z(-h,x)\leq0$ $\forall x\in\O$ and $z$ satisfies
$$0\geq\dfrac{\mathcal{L}_1W}{\varphi}=z_t-\Delta z-\dfrac{2}{\varphi}\nabla_y\varphi\cdot\nabla_y z-c\partial_1 z+\lambda_\mu z.$$
Since $\lambda_\mu>0$, the parabolic maximum principle implies that $z\leq0$ in $(-\infty,-h)\times\overline{\O}$. As a consequence, one gets
$$0<\varepsilon\leq\tilde{u}_\infty(0,0,\zeta)\leq\lim_{h\to-\infty}\upsilon(0,\zeta)=0.$$
This  contradiction  concludes the proof.
\end{proof}

The next result concerns the long time behavior of the solution of Eq. (\ref{0}) in $L^1(\O)$. The main difficulty  is to deal with the case $f_s(x,0)$ being sign-changing as $|x|$ large. To overcome it, we decompose  solution of (\ref{0}) into the sum of two integrable functions. The following lemma plays a key role.

\begin{lemma} \label{lem:l3}
Let  $w(t,x)$ be a nonnegative bounded solution of
\begin{equation}
\left\{\begin{array}{ll}
\partial_t w=\Delta w+c\partial_1 w+\zeta(t,x)w,&\textrm{$t>0$, $x\in\O$} \\
\partial_\nu w(t,x)=0,&\textrm{$t>0$, $x\in\partial\O$}
\end{array} \right.\label{V1T3.1}
\end{equation}
with initial function $w(0,\cdot)=w_0\in L^1(\O)\cap L^\infty({\O})$. We assume, in addition, that $\lim _{t \to \infty } w(t,x) = 0$, pointwise in $x\in\O$, $\zeta(t,x)\in L^\infty(\R\times\O)$ and that there exists $\mu\in L^\infty(\omega)$ satisfying
\begin{equation}
\mu(y)=\lim_{R\to\infty}\sup_{\su{t>0}{|x_1|\geq R}}\zeta(t,x_1,y), \quad\textrm {and}\quad\lambda_N(-\Delta_y-\mu(y),\omega)>0.\label{V1T3.2}
\end{equation}
There holds $$\lim_{t\to\infty}\Vert w(t,x)\Vert_{L^1(\O)}=0.$$
\end{lemma}
\begin{proof}[Proof of Lemma \ref{lem:l3}]From  Eq. (\ref{V1T3.1}), for any $\delta>0$, we have 
$$\partial_t w-\Delta w-c\partial_1 w-(\mu(y)+\delta)w=(\zeta(t,x)-\mu(y)-\delta)w.$$ 
Let us call
$$P:=\partial_t-\Delta-c\partial_1-\mu(y)-\delta,\quad\quad g(t,x):=(\zeta(t,x)-\mu(y)-\delta)w(t,x).$$
Then, we infer, from the superposition principle, that $w=w_1+w_2$, where $(w_1,w_2)$ is the solution of the system
\begin{equation}
\left\{\begin{array}{ll}
Pw_1=0 & \textrm{$t>0$, $x\in\O$}\\
\partial_\nu w_1=0 & \textrm{$t>0$, $x\in\partial\O$},
\end{array} \right.\quad\quad \left\{\begin{array}{ll}
Pw_2=g(t,x) & \textrm{$t>0$, $x\in\O$}\\
\partial_\nu w_2=0  & \textrm{$t>0$, $x\in\partial\O$},
\end{array} \right.\nonumber
\end{equation}
with the initial condition $(w_1,w_2)(0,x)=(w_0(x),0)$. From  condition (\ref{V1T3.2}), for any $\delta>0$, there exist $R>0$, such that \begin{center}
$\zeta(t,x_1,y)\leq \mu(y)+\delta$,\quad\quad  $\forall t>0,\forall(x_1,y)\in\O\setminus\O_R.$
\end{center}
and there exists an eigenpair $(\lambda_\mu,\varphi)$ of Eq. (\ref{T1.0.1}) satisfying $\lambda_\mu>0$.

Letting $\delta<\lambda_\mu$, we set $v_1(t,x)=e^{(\lambda_\mu-\delta) t}w_1(t,x)/\varphi(y)$. Then $v_1$ satisfies the equation
\begin{equation}
\left\{\begin{array}{ll}
\partial_t v_1-\Delta_x v_1-2\nabla_y v_1\cdot\dfrac{\nabla_y\varphi}{\varphi}-c\partial_1 v_1\leq0.& x\in\O.\\
\partial_\nu v_1\geq0& x\in\partial \O.\label{T3.3}
\end{array}\right.
\end{equation}
By the Hopf lemma, we know that $\inf_{y\in\omega}\varphi>0$, then $\|v_1(0,\cdot)\|_{L^\infty(\O)}\leq\|w_0\|_{L^\infty(\O)}/\inf_{\omega}\varphi$.
Then, the parabolic maximum principle yields
 $\|v_1\|_{L^\infty(\O)}\leq\|w_0\|_{L^\infty(\O)}/\inf_{\omega}\varphi$. It follows immediately that $\|w_1(t,\cdot)\|_{L^\infty(\O)}\to0$ as $t\to\infty.$
On the other hand, set $v_1^r(t,\rho,y)=\int_{-r}^r v_1(t,x_1+\rho,y)dx_1$,  we obtain
\begin{equation}
\partial_t v_1^r-\Delta v_1^r-2\dfrac{\nabla_y \varphi}{\varphi}\cdot\nabla_y v_1^r(t,\rho,y)-c\partial_\rho v_1^r\leq0\quad\quad(\rho,y)\in\O.\label{T3.4}
\end{equation}
Since $w_0(0,\cdot)\in L^1(\O)\cap L^\infty({\O})$, $v_1^r(0,\rho,y)$ is well-defined a.e on $\O$, $\forall r>0$. Moreover, there exists a constant $M$ such that  $v_1^r(0,\rho,y)\leq M$, a.e in ${\O}$, $\forall r>0$, $M$ is a supersolution of Eq. (\ref{T3.4}). Then, we infer from the parabolic comparison principle \cite{LI}, that
$$v_1^r(t,\cdot)\leq M,\quad \textrm{a.e in ${\O}$, $\forall t>0,\forall r>0$}.$$
Therefore, $\|v_1(t,\cdot)\|_{L^1(\O)}=\lim_{r\to\infty}\int_{\omega}\int_{-r}^r v_1(t,x_1+\rho,y)dx_1dy\leq M|\omega|$, $\forall t>0$. As a consequence, we get
$$\lim_{t\to\infty}\|w_1(t,\cdot)\|_{L^\infty(\O)}=\lim_{t\to\infty}\|w_1(t,\cdot)\|_{L^1(\O)}=0.$$

On the other hand, by assumption, $w$ is bounded, then $w_2$ is bounded and so it is integrable on any compact set. The same argumentation as of Proposition \ref{lem0} enables us to find supersolution of the problem satisfied by $w_2$ of the form $\xi(x_1,y)=Ce^{-{\tau}|x_1|}\varphi(y)$ such that $P(\xi)\geq0\geq g(t,x)$ in $\O\setminus\O_R$. We have $w_2(0,x)=0<\xi(x)$, $\partial_\nu w_2\leq\partial_\nu\xi$ for $y\in\partial \omega$ and the fact that $\inf_\omega\varphi>0$ allows us to find a constant $C$  such that $w_2(\pm R,y)\leq C e^{-\tau R}\varphi(y)$, $\forall y\in\omega$. Therefore, the parabolic maximum principle implies that $w_2(t,x)\leq \xi(x)$, $\forall x\in\O\setminus \overline{\O}_R$. Since $w_2$ is bounded, one can choose $C$ large enough  so that $w_2(t,x)\leq \xi(x)$, $\forall x\in\O$. Moreover,
$$\forall x\in\O\quad\quad\lim_{t\to\infty} w_2^+(t,x)=\lim_{t\to\infty}(w-w_1)^+(t,x)=0.$$
Hence $0\leq w= w_1+w_2\leq w_1+w^+_2$, which is integrable on $\O$. It follows from Lebesgue's dominated convergence theorem that  $\lim _{t \to \infty } \|w(t,x)\|_{L^1(\O)} = 0$  because  $\lim _{t \to \infty } w(t,x) = 0$  for $x\in\O$. We thus conclude  the proof.
\end{proof}
We are now in a position to prove  \thm{T3}
\begin{proof}[\textbf{Proof of \thm{T3}}] The proof is a direct consequence of  Lemma \ref{lem:l3}.
Let $u$ be the solution of $(\ref{0})$ with $u(0,x)=u_0(x)\in L^\infty(\O)\cap L^1(\O)$. The function $\tilde{u}(t,x_1,y):=u(t,x_1+ct,y)$ satisfies  Eq. $(\ref{V1T3.1})$ with the same initial condition $u_0$. Let $W$ be defined as following :
\begin{equation}
W(x)=\left\lbrace\begin{array}{ll}
0&\textrm{if $c\geq c^*$}\\
U(x) & \textrm{if $c<c^*$,}
\end{array}\right.
\end{equation}
where $U(x)$ is the unique positive solution of Eq. (\ref{3}) when $c<c^*$.

Let $\overline{u}$,  $\underline{u}$ be respectively the solutions of $(\ref{V1T3.1})$ with initial conditions $\overline{u}(0,x)=\max\lbrace u_0(x),W(x)\rbrace$ and $\underline{u}(0,x)=\min\lbrace u_0(x),W(x)\rbrace$.  We know, from \thm{T2} that the functions $\overline{u}(t,x_1-ct,y)$ and $\underline{u}(t,x_1-ct,y)$ converge to $W(x)$ as $t\to\infty$, uniformly with respect to $x\in\O$. Moreover, the parabolic maximum principle yields
$$\forall t>0,x\in\O\quad \overline{u}(t,x)\geq\max\lbrace \tilde{u}(t,x),W(x)\rbrace\quad \underline{u}(t,x)\leq\min\lbrace \tilde{u}(t,x),W(x)\rbrace.$$
Therefore, the functions $\overline{w}(t,x):=\overline{u}(t,x)-W(x)$ and $\underline{w}(t,x):=W(x)-\underline{u}(t,x)$ is nonnegative bounded solution of Eq. $(\ref{V1T3.1})$ with
$$\overline{\zeta}(t,x)=\dfrac{f(x,\overline{u})-f(x,W)}{\overline{u}-W}\quad ;\quad \underline{\zeta}(t,x)=\dfrac{f(x,W)-f(x,\underline{u})}{W-\underline{u}}. $$
Thanks to condition (\ref{5}), one easily sees that $\overline{\zeta}$ and $\underline{\zeta}$ are less than $f_s(x,0)$. Thanks to  condition (\ref{6}), $\overline{\zeta}$ and $\underline{\zeta}$ satisfy (\ref{V1T3.2}). The initial conditions $\overline{w}(0,x), \underline{w}(0,x)\in L^1(\O)$ allow one to apply Lemma \ref{lem:l3} to derive 
$$\lim_{t\to\infty}\Vert \overline{u}-W\Vert_{L^1(\O)}=0\quad ;\quad \lim_{t\to\infty}\Vert W-\underline{u}\Vert_{L^1(\O)}=0.$$
This completes the proof because $\underline{u}\leq \tilde{u}\leq \overline{u}.$
\end{proof}

The next section is of independent interest. We are concerned with the existence, uniqueness, long time behavior of pulsating fronts, which are T-periodic in $t$ and  periodic in $y$.

\section{The partially periodic environment with time dependence}

Before proving the main results, let us introduce some new definitions  and preliminary results that are needed in this section.

\begin{proposition}\label{pro:p2}Let $\mathcal{O}_r=\R\times(-r,r)\times\R^{N-1}$, then for any $r>0$, there exists a  real number $\lambda_{p}(r)$ and $\chi_r(t,x)\in C^{1,2}_{t,x}(\R\times(-r,r)\times\R^{N-1})$,  solving the eigenvalue problem
\begin{equation}
\left\{\begin{array}{ll}
\mathcal{P}\chi_r=\lambda_{p}(r)\chi_r &\textrm{a.e in $\mathcal{O}_r$}\\
\chi_r=0 & \textrm{on $\partial\mathcal{O}_r$}\\
\textrm{$\chi_r$ is periodic both in $y$ and $t$},
\end{array} \right.\nonumber
\end{equation}
where $\mathcal{P}$ is the parabolic operator defined in (\ref{12.1}) and $f$ is periodic in $y$ and $t$. Moreover, as $r\to\infty$, $\lambda_{p}(r)$ decreasingly converges to $\lambda_p$ ,   where $\lambda_p=\tilde{\lambda}_1(\mathcal{P},\R\times\mathbb{R}^N)$ defined in (\ref{13}) and there exists an  eigenfunction $\chi\in C^{1,2}_{t,x}(\R\times\R^{N})$ associated with $\tilde{\lambda}_{1}$ such that
$\mathcal{L}\chi=\tilde{\lambda}_{1}\chi$ a.e in $\R^{N+1}$.
\end{proposition}
\begin{proof} To prove the existence of $\lambda_p(r)$, we consider the eigenvalue problem
\begin{equation}
\left\{\begin{array}{ll}
\mathcal{P}\chi_{r,\rho}=\lambda_{p}({r,\rho})\chi_{r,\rho} &\textrm{a.e in $\mathcal{O}_{r,\rho}$}\\
\chi_{r,\rho}=0 & \textrm{on $\partial\mathcal{O}_{r,\rho}$}\\
\textrm{$\chi_{r,\rho}$ is periodic both in $t$},
\end{array} \right.\nonumber
\end{equation}
where $\mathcal{O}_{r,\rho}=\R\times(-r,r)\times B_\rho$. It has been proved (see \cite{HE}) that the eigenvalue $\lambda_{p}({r,\rho})$ is well-defined and unique
and that $\chi_{r,\rho}$ is unique up to multiplicative constant. Thanks to availability of Krylov-Safonov-Harnack inequality and Schauder estimates for parabolic operator, we can use the same scheme of the arguments in the proof of Theorem 2.16, 2.17 in \cite{NA1} to show that $\lambda_{p}({r,\rho})\searrow\lambda_{p}({r})$ as $\rho\to\infty$ and $\lambda_{p}({r})\searrow\lambda_p$ as $r\to\infty$. The existence of $\chi_{r,\rho}$ and $\chi$ respectively associated with  $\lambda_{p}({r,\rho})$ and $\lambda_p$ is followed. We omit the details.
\end{proof}

We are now able to prove  Theorem \ref{thm:T4}.

\begin{proof}[\textbf{Proof of \thm{T4}}]
Recall that $\tilde{\lambda}_1=\tilde{\lambda}_1(\mathcal{P},\R\times\R^N)$. We  consider the first case $\tilde{\lambda}_1<0$. It  follows from Proposition \ref{pro:p2} that for $r>0$  large enough, $\lambda_p(r)<0$. Let $\chi_r$ be an eigenfunction associated to $\lambda_p(r)$ of $\mathcal{P}$ in $\mathcal{O}_r$, we define the function :
\begin{equation}
\phi(t,x)=\left\{\begin{array}{ll}
 \eta\chi_r(t,x) &\textrm{$x\in\mathcal{O}_r$}\\
0 & \textrm{otherwise}.\\
\end{array} \right.\nonumber
\end{equation}
For $\eta\in\R$ small enough, one obtains immediately that
$$\partial_t\phi-\Delta\phi-c\partial_1\phi-f(t,x,\phi)=(f_s(t,x,0)+\lambda_p(r))\phi_R-f(t,x,\phi)<0.$$
That is, $\phi$ is a subsolution of Eq. (\ref{12}) while the constant $S$ given in (\ref{9}) is a super solution of Eq. (\ref{12}).     Let us consider the solution $u$ of Eq. (\ref{12}) with the initial condition $u(0,x)=\phi(0,x)$. The standard parabolic theory and maximum principle imply that there exists such a solution for any $t>0$ and satisfies $\phi(t,x)\leq u(t,x)\leq S$, $\forall (t,x)\in\R\times\R^{N+1}$. In particular, $\phi(T,x)\leq u(T,x)$, where $T$ is the period of $\phi$ and $f$ with respect to $t$. Consider the function $u(t+T,x)$; it is also a solution of Eq. (\ref{12}) with initial condition $u(T,x)\geq u(0,x)$, then $u(t+T,x)\geq u(t,x)$. In particular, $u(2T,x)\geq u(T,x)$. By induction, one sees that the sequence $u_n(t,x)=u(t+nT,x)$ is nondecreasing in $n$ and uniformly bounded  by $S$. Therefore, $u_n(t,x)$ converges pointwise to a bounded function $U(t,x)$ such that $U$ is T-periodic in $t$, $\phi\leq U\leq S$, $U$ solves Eq. (\ref{12}). The partial periodicity in $y$ of solution follows from the construction.

Let us postpone for a moment the proof of the necessary condition to prove the uniqueness of the solution. We emphasize that the uniqueness of (\ref{12}) is proved to hold in the class of positive bounded solutions without a-priori assuming to be periodic in $y$ nor in $t$.

Assume by contradiction that $\overline{U}$ and $\underline{U}$ are two positive bounded solutions of  (\ref{12}), Theorem 5.2, Appendix, yields
$$\lim_{|x_1|\to\infty}\overline{U}(t,x_1,y)=\lim_{|x_1|\to\infty}\underline{U}(t,x_1,y)=0,$$
uniformly in $y$ and $t$. By condition (\ref{11}), there exists a unique eigenpair $(\lambda,\varphi)$, $\lambda>0$ satisfying 
\begin{equation}
\left\lbrace\begin{array}{ll}
\partial_t\varphi-\Delta\varphi-\gamma(t,y)\varphi=\lambda\varphi\\
\varphi>0\\
\varphi(.,.+T)=\varphi\\
\varphi(.+L_ie_i,.)=\varphi
\end{array}\right.\quad\quad\textrm{in $\R\times\R^{N-1}$, $i\in\{1,N-1\}$}.\label{T4.0}
\end{equation}
Thanks to the periodicity, we have $\inf_{\R^{N}}\varphi(t,y)>0$. For any $\varepsilon>0$, there exists $R(\varepsilon)>0$ such that
\begin{equation}
\underline{U}(t,x_1,y)\leq\varepsilon\varphi(t,y),\quad\quad\forall\, |x_1|\geq R(\varepsilon),y\in\R^{N-1},t\in\R\label{T4.1},
\end{equation}
and therefore, the set 
$$K_\varepsilon:=\{k>0:k\overline{U}\geq\underline{U}-\varepsilon\varphi \textrm{ in } {\R\times\R^N} \}$$
is nonempty. Set $k(\varepsilon):=\inf K_\varepsilon$. Obviously, the function $k(\varepsilon):\mathbb{R}^+\to\R$ is nonincreasing. Assume by way of contradiction that $$k^*=\mathop{\lim}\limits_{\varepsilon\to0^+}k(\varepsilon)>1.$$
Note that $k^*$ could be $\infty$. Take $0<\varepsilon<\sup_{\R^{N+1}} \underline{U}/\varphi$, we see that $k(\varepsilon)>0$, $k(\varepsilon)\overline{U}-\underline{U}+\varepsilon\varphi\geq0$. The definition of $k(\varepsilon)$ yields that there exists a sequence $(t_n^\varepsilon, x^\varepsilon_{1,n},y^\varepsilon_{n})$ in $\R\times\R^N$ such that
\begin{equation}\label{V1_T4.1.0}
\left(k(\varepsilon)-\dfrac{1}{n}\right)\overline{U}(t_n^\varepsilon,x^\varepsilon_{1,n},y^\varepsilon_{n})<\underline{U}(t_n^\varepsilon,x^\varepsilon_{1,n},y^\varepsilon_{n})-\varepsilon\varphi(t_n^\varepsilon,y^\varepsilon_{n}).
\end{equation}
From (\ref{T4.1}), we have $(t_n^\varepsilon, x^\varepsilon_{1,n},y^\varepsilon_n)\in\mathcal{O}_{R(\varepsilon)}$ for $n$ large enough. Taking the sequences $(\tau_n^\varepsilon)$ and $(z_n^\varepsilon)$  such that $t_n^\varepsilon-\tau_n^\varepsilon\in[0,T)$ and $y_n^\varepsilon-z_n^\varepsilon\in[0,L_2)\times...\times[0,L_{N})$. For any $\varepsilon>0$, one sees that $\overline{U}_n^\varepsilon(t,x_1,y)=\overline{U}(t+\tau_n^\varepsilon,x_1,y+z_n^\varepsilon)$ and $\underline{U}_n^\varepsilon(t,x_1,y)=\underline{U}(t+\tau_n^\varepsilon,x_1,y+z_n^\varepsilon)$ are  solutions of the following equation
$$\partial_t U_n^\varepsilon-\Delta U_n^\varepsilon-c\partial_1 U_n^\varepsilon=f(t+\tau_n^\varepsilon,x_1,y+z_n^\varepsilon,U_n^\varepsilon).$$
Using the priori estimates of solutions (Theorem 5.2, Appendix), we deduce that as $n\to\infty$, up to extractions, $\overline{U}_n^\varepsilon\to \overline{U}_\infty^\varepsilon$ and $\underline{U}_n^\varepsilon\to \underline{U}_\infty^\varepsilon$ locally uniformly in $\R^{N+1}$. Moreover, since $f$ is  periodic  in $y$ and T-periodic in $t$, there exist $\tau_\infty^\varepsilon$, $z_\infty^\varepsilon$ such that  $\overline{U}_\infty^\varepsilon$ and $\underline{U}_\infty^\varepsilon$ satisfy, by the standard parabolic estimates, the equation
\begin{equation}
\partial_t U_\infty^\varepsilon-\Delta U_\infty^\varepsilon-c\partial_1 U_\infty^\varepsilon=f(t+\tau_\infty^\varepsilon,x_1,y+z_\infty^\varepsilon,U_\infty^\varepsilon).\nonumber
\end{equation}
One has $W_\infty^\varepsilon=k(\varepsilon)\overline{U}_\infty^\varepsilon-\underline{U}_\infty^\varepsilon+\varepsilon\varphi_\infty^\varepsilon\geq0$, where $\varphi_\infty^\varepsilon(t,y)=\lim_{n\to\infty}\varphi(t+\tau_n^\varepsilon, y+z_n^\varepsilon)$ satisfying
$\partial_t\varphi_\infty^\varepsilon-\Delta\varphi_\infty^\varepsilon-\gamma_\infty^\varepsilon(t,y)\varphi_\infty^\varepsilon=\lambda\varphi_\infty^\varepsilon$,  with $\gamma^\varepsilon_\infty(t,y)=\gamma(t+\tau_\infty^\varepsilon,y+z_\infty^\varepsilon)$, $\phi_\infty^\varepsilon$ is periodic in $y$ and $t$. Moreover, by passing to the limit in (\ref{V1_T4.1.0}), one finds $(t(\varepsilon),x_1(\varepsilon),y(\varepsilon))$ such that
\begin{equation}
(k(\varepsilon)\overline{U}_\infty^\varepsilon-\underline{U}_\infty^\varepsilon+\varepsilon\varphi_\infty^\varepsilon)(t(\varepsilon),x_1(\varepsilon),y(\varepsilon))=0.\label{T4.2}
\end{equation}
Note that $t(\varepsilon)\in [0,T)$ and $y(\varepsilon)\in[0,L_2)\times...\times[0,L_{N})$ are bounded with respect to $\varepsilon$. The case that $\mathop {\lim\inf }\limits_{\varepsilon  \to 0^+}|x_1(\varepsilon)|<\infty$ is ruled out. Indeed, if $\mathop {\lim\inf }\limits_{\varepsilon  \to 0^+}|x_1(\varepsilon)|<\infty$, there exists a sequence $(\varepsilon_n)\to0$ as $n\to\infty$ such that $(t(\varepsilon_n),x_1(\varepsilon_n),y(\varepsilon_n))\to(t_0,x_0,y_0)$ as $n\to\infty$, up to subsequences. Moreover, by the partial periodicity in $y$  and the periodicity in $t$ of $f$, the standard parabolic estimates yield that $\overline{U}^{\varepsilon_n}_\infty\to \overline{U}^0_\infty $ and $\underline{U}^{\varepsilon_n}_\infty\to \underline{U}^0_\infty $ locally uniformly in $\R^{N+1}$ satisfying the following equation
\begin{equation}
\partial_t U_\infty^0-\Delta U_\infty^0-c\partial_1 U_\infty^0=f(t+t_\infty^0,x_1,y+z_\infty^0,U_\infty^0).\nonumber
\end{equation}
for some $t_\infty^0\in[0,T)$ and $z_\infty^0\in[0,L_2)\times...\times[0,L_{N})$. From (\ref{T4.2}) $k^*<\infty$, then  the function $W=k^*\overline{U}_\infty^0-\underline{U}_\infty^0$ is nonnegative and vanishes at $(t_0,x_0,y_0)$. Since $k^*>1$, The Lipschitz continuity of $f$ with respect to $s$ and condition (\ref{10}) yield
\begin{eqnarray}
\partial_t W-\Delta W-c\partial_1 W&\geq& k^*f(t+t_\infty^0,x_1,y+z_\infty^0,\overline{U}_\infty^0)-f(t+t_\infty^0,x_1,y+z_\infty^0,\underline{U}_\infty^0)\nonumber\\
&\geq& f(t+t_\infty^0,x_1,y+z_\infty^0,k^*\overline{U}_\infty^0)-f(t+t_\infty^0,x_1,y+z_\infty^0,\underline{U}_\infty^0)\nonumber\\
&\geq& z(t+t_\infty^0,x_1,y+z_\infty^0)W,\label{T4.21}
\end{eqnarray}
where $z(t,x)\in L^\infty_{\textrm{loc}}(\R\times\R^N)$. Hence, the parabolic strong maximum principle implies $W=0$ in $(-\infty,t^0_\infty)\times\R^N$. This is a contradiction because from  condition (\ref{10}), the last inequality at (\ref{T4.21}) hold strictly in some $D\subset(-\infty,t^0_\infty)\times\R^N$ with $|D|>0$. It remains to consider the case $\mathop {\lim}\limits_{\varepsilon  \to 0^+}|x_1(\varepsilon)|=\infty$. We have shown that $W_\infty^\varepsilon\geq0$ and $W_\infty^\varepsilon$ vanishes at $(t(\varepsilon),x_1(\varepsilon),y(\varepsilon))$, we infer that there exists a neighborhood $\mathcal{O}$ 
of $(t(\varepsilon),x_1(\varepsilon),y(\varepsilon))$ such that $k(\varepsilon)\overline{U}<\underline{U}$ in $\mathcal{O}$, shrinking $\mathcal{O}$ if necessary. Since $k^*>1$, for $\varepsilon$ small enough, $k(\varepsilon)>1$, we derive from (\ref{10}) for $x\in\mathcal{O}$
\begin{eqnarray}
\partial_tW_\infty^\varepsilon&-&\Delta W_\infty^\varepsilon-c\partial_1 W_\infty^\varepsilon\nonumber\\
&\geq& k(\varepsilon)f(t+\tau_\infty^\varepsilon,x_1,y+z_\infty^\varepsilon,\overline{U}_\infty^\varepsilon)-f(t+\tau_\infty^\varepsilon,x_1,y+z_\infty^\varepsilon,\underline{U}_\infty^\varepsilon)+(\gamma(t+\tau_\infty^\varepsilon,y+z_\infty^\varepsilon)+\lambda)\varepsilon\varphi_\infty^\varepsilon\nonumber\\
&\geq&  f(t+\tau_\infty^\varepsilon,x_1,y+z_\infty^\varepsilon,k(\varepsilon)\overline{U}_\infty^\varepsilon)-f(t+\tau_\infty^\varepsilon,x_1,y+z_\infty^\varepsilon,\underline{U}_\infty^\varepsilon)+(\gamma(t+\tau_\infty^\varepsilon,y+z_\infty^\varepsilon)+\lambda)\varepsilon\varphi_\infty^\varepsilon\nonumber\\
&\geq&\dfrac{f(t+\tau_\infty^\varepsilon,x_1,y+z_\infty^\varepsilon,k(\varepsilon)\overline{U}_\infty^\varepsilon)}{k(\varepsilon)\overline{U}_\infty^\varepsilon}(k(\varepsilon)\overline{U}_\infty^\varepsilon-\underline{U}_\infty^\varepsilon+\varepsilon\varphi_\infty^\varepsilon)+\dfrac{\lambda}{2}\varepsilon\varphi_\infty^\varepsilon\nonumber\\
&+&\left(\dfrac{\lambda}{2}+\gamma(t+\tau_\infty^\varepsilon,y+z_\infty^\varepsilon)-\dfrac{f(t+\tau_\infty^\varepsilon,x_1,y+z_\infty^\varepsilon,k(\varepsilon)\overline{U}_\infty^\varepsilon)}{k(\varepsilon)\overline{U}_\infty^\varepsilon}\right)\varepsilon\varphi_\infty^\varepsilon. \label{T4.3}
\end{eqnarray}
Condition  $(\ref{11})$ yields, for $\varepsilon$ small enough, that
$$\dfrac{f(t+\tau_\infty^\varepsilon,x_1,y+z_\infty^\varepsilon,k(\varepsilon)\overline{U}_\infty^\varepsilon)}{k(\varepsilon)\overline{U}_\infty^\varepsilon}<\gamma^\varepsilon_\infty(t,y)+\dfrac{\lambda}{2},\quad\quad\forall(x_1,y)\in\mathcal{O}.$$
Then, it follows from (\ref{T4.3}) that
$$\partial_t W_\infty^\varepsilon-\Delta W_\infty^\varepsilon-c\partial_1 W_\infty^\varepsilon-\varrho(t+\tau_\infty^\varepsilon,x_1,y+z_\infty^\varepsilon)W_\infty^\varepsilon> \dfrac{\lambda}{2}\varepsilon\varphi_\infty^\varepsilon>0\quad\textrm{in $\mathcal{O}$},$$
where $\varrho(t,x)=\frac{f(t,x,k(\varepsilon)\overline{U}_\infty^\varepsilon)}{k(\varepsilon)\overline{U}_\infty^\varepsilon}$ is bounded. This is a contradiction because the strong maximum principle asserts that $W^\varepsilon_\infty(t,x_1,y)=0$ in $\mathcal{O}$. As a consequence, we have proved that $k^*=\mathop{\lim}\limits_{\varepsilon\to0^+}k(\varepsilon)\leq 1$. Therefore
$$\underline{U}\leq\mathop{\lim}\limits_{\varepsilon\to0^+}(k(\varepsilon)\overline{U}+\varepsilon\varphi)\leq \overline{U}\quad\textrm{in $\R^{N+1}$.}$$
We derive the uniqueness by exchanging the role of $\overline{U}$ and $\underline{U}$. Note that we do not use the  periodicity of $y$ and  $t$ of solution in  the proof of uniqueness. This thus implies that any positive bounded solutions of Eq. (\ref{12}) must be  periodic in $y$ and T-periodic in $t$.

To conclude the proof of Theorem \ref{thm:T4}, it only remains to prove the necessary condition. Assume by contradiction that $\tilde{\lambda}_1\geq0$ and  Eq. (\ref{12}) admits a solution $U$, which is T-periodic in $t$ but not necessarily periodic in $y$. Let $\chi$ be a principal eigenfunction associated with $\tilde{\lambda}_1$ (Proposition \ref{pro:p2}) with normalization $\chi(0,0)<U(0,0)$. Then
$$\partial_t\chi-\Delta\chi-c\partial_1\chi-f(t,x,\chi)=\tilde{\lambda}_1\chi+f_s(t,x,0)\chi-f(t,x,\chi)\geq0.$$
Arguing similarly as the proof of uniqueness, we achieve the  contradiction : $U\leq\chi$ in $\R^{N+1}$. 
\end{proof}
Before investigating the long time behavior, we point out that the monotonicity in time of solutions starting by a stationary sub (or super) solution of parabolic operator with time-dependent coefficients no longer holds. In addition, the boundedness of initial datum does not suffice to guarantee that the solutions of Eq. (\ref{8}) converge uniformly to the unique solution of Eq. (\ref{12}) as $t\to\infty$. However, thanks to the periodicity in $t$ of solutions, which obtained by the uniqueness, we will have the locally uniform convergence and under some extra conditions (part (ii), Theorem \ref{thm:T5}) we can actually  derive the uniform convergence as $t\to\infty$.

\begin{proof}[\textbf{Proof of Theorem \ref{thm:T5}}]
Set $S':=\max\{S,\|u_0\|_{L^\infty(\O)}\}$, $S$ is the positive constant given in (\ref{9}). Then, the function $\tilde{u}(t,x)=u(t,x_1+ct,y)$ satisfies $0<\tilde{u}\leq S' $ in $\R^+\times\R^N$ and solves
\begin{equation}
\partial_t\tilde{u}=\Delta\tilde{u}+c\partial_1\tilde{u}+f(t,x,\tilde{u})\quad t>0,x\in\R^N,\label{T5.1}
\end{equation}
with initial condition $\tilde{u}(0,u)=u_0(x)$. Let $w$ be the solution to (\ref{T5.1}) with initial condition $w_0(x)=S'$. Clearly, the constant $S'$ is T-periodic in $t$ and periodic in $y$. Arguing as the proof of Theorem \ref{thm:T4}, we deduce that the sequence $w_n(t,x)=w(t+nT,x)$ is nonincreasing and converges locally uniformly to $W(t,x)$, which is a solution of 
\begin{equation}
\partial_t W-\Delta W-c\partial_1 W-f(t,x,W)=0\quad\textrm{$\forall t>0$, $x\in\R^N$}.\label{T5.2}
\end{equation}
Moreover, $W(t,x)$ is $T$ periodic in $t$ and periodic in $y$. Then
\begin{equation}
\forall r>0,\quad\lim_{t\to\infty}\sup_{x\in\mathcal{O}_r}(\tilde{u}(t,x)-W(t,x))\leq\lim_{t\to\infty}\sup_{x\in\mathcal{O}_r}(w(t,x)-W(t,x))=0.\nonumber
\end{equation}
The following claim holds true

\textbf{Claim.}\begin{equation}
\lim_{\min(t,|x_1|)\to\infty}\tilde{u}(t,x_1,y)=0\quad\textrm{uniformly in $y\in\R^{N-1}$.}\label{T5.3}
\end{equation}
Let us postpone for a moment the proof of claim to proceed the proof.

If $\tilde{\lambda}_1\geq0$, then $W\equiv0$ in $\R\times\R^N$. Therefore, $\tilde{u}(t,x)\to0$ as $t\to\infty$ locally uniformly with respect to $x\in\R^N$. This convergence is uniform in $x\in\R^N$ due to (\ref{T5.3}).

If $\tilde{\lambda}_1<0$. From Propositions \ref{pro:p2}, there exists $\rho>0$ such that $\lambda_{p}(\rho)<0$. Let $\chi_\rho(t,x)$ is an associated principal eigenfunction to $\lambda_{p}(\rho)$,  for $k>0$ small enough, one sees that the function
 \begin{equation}
V(t,x)=
\left\{\begin{array}{ll}
k\chi_\rho(t,x) &\textrm{ $x\in \mathcal{O}_\rho$}\\
0 & \textrm{otherwise}\\
\end{array} \right.\nonumber
\end{equation}
is a subsolution of Eq. (\ref{T5.2}). Then if (\ref{T5.0}) holds, one can choose $k>0$ in such the way $\inf_{\mathcal{O}_\rho}u_0(x)>k\sup\chi_\rho(0,x)>0$, namely $V(0,x)\leq u_0(x)$. Alternatively, $u_0(x)$ is  periodic in $y$, then $\tilde{u}(t,x)$ is strictly positive,  periodic in $y$ and T-periodic in $t$, then the parabolic strong maximum principle yields $V(T,x)\leq\tilde{u}(T,x)$. In both case, we always can define $\tilde{v}(t,x)$ is such that  $\tilde{v}(0,x)=V(0,x)$ or $\tilde{v}(T,x)=V(T,x)$. It follows immediately by parabolic maximum principle that $\tilde{v}(t,x)\leq u(t,x)$, $\forall t>T,x\in\R^N$. Arguing similarly as the proof of Theorem \ref{thm:T4}, we deduce that the sequence $v_n(t,x)=\tilde{v}(t+nT,x)$ is nondecreasing and converges locally uniformly to $P(t,x)$, which is a solution of 
\begin{equation}
\partial_t P-\Delta P-c\partial_1 P-f(t,x,P)=0\quad\textrm{$\forall t>0$, $x\in\R^N$}.\nonumber
\end{equation}
Moreover, $P$ is strictly positive,  periodic in $y$ and T-periodic in $t$. The uniqueness of Theorem \ref{thm:T4} implies that $W=P=U$ in $\R^{N+1}$, which is a solution of Eq. (\ref{12}). Assume by contradiction that this convergence is not uniform in $x$, this means that there exist $\varepsilon>0$ and  a sequence $(t_n,x_{1,n},y_n)\in\R^+\times\R^N$ such that 
\begin{equation}\label{T5.4}
\lim_{n\to\infty}t_n=\infty,\quad\forall n\in\N,\quad |\tilde{u}(t_n,x_{1,n},y_n)-U(x_{1,n},y_n)|\geq\varepsilon. 
\end{equation}
Due to the locally uniform convergence of $\tilde{u}$, necessarily, the sequence $(x_{1,n})$ is unbounded. We get, by the a priori estimates of $U$, that $U(x_{1,n},y_n)\to0$ as ${n\to\infty}$, uniformly in $y$. But from (\ref{T5.4}), this inference contradicts  the claim (\ref{T5.3}). Thus,  to conclude the proof, it remains to prove the Claim \ref{T5.3}.

Let us call $(t_n,x_{1,n},y_n)\in\R^+\times\R^N$ be a sequence such that the claim (\ref{T5.3}) is not true :
$$\lim_{n\to\infty}t_n=\lim_{n\to\infty}x_{1,n}=\infty,\quad\forall n\in\N,\quad \liminf_{n\to\infty}\tilde{u}(t_n,x_{1,n},y_n)\geq\varepsilon\quad\textrm{for some $\varepsilon>0$}.$$
For any $n\in\mathbb{N}$, we define the functions $\tilde{u}_n(t,x)=\tilde{u}(t+t_n,x_1+x_{1,n},y+y_n)$. It holds that $0\leq \tilde{u}_n\leq S'$ and
$$\partial_t\tilde{u}_n=\Delta\tilde{u}_n+c\partial_1\tilde{u}_n+f(t+t_n,x_1+x_{1,n},y+y_n,\tilde{u}_n)\quad\quad\textrm{$t>-t_n$, $x\in\R^N$}.$$
Since $f$ is  periodic in $y$ and T-periodic in $t$, we can assume without loss of generality that $t_n\to t_0$ and $y_n\to y_0$ (up to a subsequence) as $n\to\infty$. Thanks to (\ref{10})-(\ref{11}), we deduce, by the parabolic estimates and embedding theorems   that $\tilde{u}_n$ converges (up to a subsequence) to $\tilde{u}_\infty$ locally uniformly in $\R\times\R^N$  satisfying 
$$\partial_t\tilde{u}_\infty\leq\Delta\tilde{u}_\infty+c\partial_1\tilde{u}_\infty+\gamma(t+t_0,y+y_0)\tilde{u}_\infty\quad\quad\textrm{$\forall t\in\R$, $x\in\R^N$}$$
and $\tilde{u}_\infty(0,0)\geq\varepsilon$. Moreover,  condition (\ref{11}) implies that there exists an  eigenpair $(\lambda,\varphi)$, $\lambda>0$ verifying (\ref{T4.0}). Now, if one sets $\gamma^0(t,y)=\gamma(t+t_0,y+y_0)$, $\varphi^0(t,y)=\varphi(t+t_0,y+y_0)$ and $\upsilon^0(t,x)=S'\varphi^0(t,y)e^{-\lambda(t-h)}$, we have $\partial_t \varphi^0-\Delta\varphi^0-\gamma^0\varphi^0=\lambda\varphi^0$ for $(t,y)\in\R^N$ and
$$\partial_t\upsilon^0=\Delta\upsilon^0+c\partial_1\upsilon^0+\gamma^0\upsilon^0\quad\quad\textrm{$(t,x)\in\R^{N+1}$}.$$
Let $\mathcal{\widetilde{L}}=\partial_t-\Delta-c\partial_1-\gamma^0$, $W^0(t,x)=\upsilon^0-\tilde{u}_\infty$ and $W^0(t,x)=z(t,x)\varphi^0(t,y)$. One has
% $W^{0}(t,x)=z^{0}(t,x)\varphi^0(t,y)$, where $W^0(t,x_1,y)=W(t+t_0,x_1,y+y_0)$, $z^0(t,x_1,y)=z(t+t_0,x_1,y+y_0)$,  and . Therefore,
%\begin{eqnarray}
%\mathcal{\widetilde{ L}}[W(t+r_0,x_1,y+y_0)]&=&z_t(t+t_0,x_1,y+y_0)\varphi(t+t_0,y+y_0)+z(t+t_0,x_1,y+y_0)\varphi_t(t+t_0,y+y_0)\nonumber\\
%&-&\varphi(t+t_0,y+y_0)\Delta z(t+t_0,x_1,y+y_0)-z(t+t_0,x_1,y+y_0)\Delta \varphi(t+t_0,y+y_0)
%\end{eqnarray}

$$0\leq\dfrac{\mathcal{\widetilde{L}}W}{\varphi^0}=z_t-\Delta z-\dfrac{\nabla_y\varphi^0}{\varphi^0}\cdot \nabla_y z-c\partial_1z+\lambda z.$$
The periodicity yields $\inf_{\R^{N}}\varphi^0(t,y)>0$. Since $\tilde{u}\leq S$, one can enlarge $S'$  so that $z(h,x)\geq0$ in $\R^N$. The parabolic comparison principle yields $z(t,x)\geq0$, $\forall t\leq h,x\in\R^N$. As a consequence
$$\varepsilon\leq \tilde{u}(0,0)\leq \upsilon^0(0,0)=S'\varphi^0(0,0)e^{\lambda h}=S'\varphi(t_0,y_0)e^{\lambda h}.$$
Letting $h$ to $-\infty$, we get a contradiction and this concludes the proof.

\end{proof}

\section{Further results  and applications}

\subsection{Similarity of the problem with Dirichlet boundary condition}
In this subsection, we aim at proving an analogous result of Theorem \ref{thm:T1}, where  Dirichlet instead of Neumann condition is imposed on the boundary of $\O$. This is  to prepare for the main goal in the next subsection. Let us first define the  generalized Dirichlet  principal eigenvalue 
\begin{eqnarray}
\lambda_{D}(-\mathcal{L},\O):=\sup\{\lambda\in\R:\exists\phi\in W^{2,N}_{loc}(\O),\phi>0,(\mathcal{L}+\lambda)\phi\leq0\textrm{ a.e in $\O$ and $\phi=0$ on $\partial\O$}\}.\label{add1.3}
\end{eqnarray}
Note that, if $\O$ is bounded, $\lambda_D$ coincides with the classical Dirichlet eigenvalue (see \cite{BNV}). Similar to  assumption (\ref{6}), we assume that there exists a measurable bounded function $\mu:\omega\to\R$  such that
\begin{equation}
\mu(y)=\mathop {\limsup}\limits_{|x_1 | \to \infty } f_s(x_1,y,0)\quad\quad\textrm{and}\quad\quad\lambda_D(-\Delta_y-\mu(y),\omega)>0.\label{4.1.0}
\end{equation}
Then, we obtain an analogue of Theorem \ref{thm:T1} as following

\begin{theorem}\label{add1_0}
Assume that conditions (\ref{4})-(\ref{5}) and (\ref{4.1.0}) are satisfied. The equation 
\begin{equation}
\left\{\begin{array}{ll}
  \Delta U+c\partial_1 U+f(x,U)=0 &\textrm{in $\O$}\\
   U=0 & \textrm{on $\partial\O$}\\
  \textrm{$U>0$ in $\O$}\\
  \textrm{$U$ is bounded.} 
\end{array} \right.\label{4.1.1}
\end{equation}
admits a solution if and only if $0\leq c<c^*$, where $c^*$ is the critical  speed given by
\begin{equation}\nonumber
c^*:=2\sqrt{-\lambda_D(-\Delta-f_s(x,0),\O)}, \quad\quad\textrm{if $\lambda_D(-\Delta-f_s(x,0),\O)<0$}.
\end{equation}
Moreover, the solution is unique when it exists.
\end{theorem}

\begin{proof}The proof of this theorem is essentially  similar to the proof of Theorem \ref{thm:T1}. However, there are significant differences to be outlined here:

i) Existence.

Since the problem is set up with the Dirichlet boundary condition,  Proposition 1, in \cite{BR2}, cannot be applied. However, the Dirichlet boundary condition allows us to use  Theorem 1.9, \cite{BR3} to prove the existence of Eq. (\ref{4.1.1}). Indeed, let us call $\lambda_D=\lambda_D(-\Delta-c\partial_1-f_s(x,0),\O)$, then arguing similarly to Proposition \ref{pro:p1}, we have  $\lambda_D<0$ iff $0\leq c<c^*$.  Let $(\lambda_R,\varphi_R)$ be the Dirichlet principal eigenvalue and eigenfunction of the problem
\begin{eqnarray}
\left\{\begin{array}{ll}
-\Delta\varphi_R-c\partial_1\varphi_R-f_s(x,0)\varphi_R=\lambda_R\varphi_R& x\in\O_R\\
\varphi_R(x_1,y)=0 & |x_1|<R, y\in\partial\omega\\
\varphi_R(\pm R,y)=0& y\in\omega,
\end{array} \right.\nonumber
\end{eqnarray}
we deduce, by Theorem 1.9 in \cite{BR3}, that $\lambda_D=\lim_{R\to\infty}\lambda_R$. Note that Theorem 1.9 in \cite{BR3} also deals with the case of nonsmooth domain as the set $\O_R$ of ours. Then the existence of  Eq. (\ref{4.1.1}) can be obtained in the same way with  Theorem \ref{thm:T1}.

ii) Nonexistence and Uniqueness. 

Let $\tilde{U}(x)=U(x)e^{\frac{c}{2}x_1}$, then $U$ solves  Eq. (\ref{4.1.1}) if and only if  $\tilde{U}$ solves the equation
\begin{equation}
\left\{\begin{array}{ll}
  \Delta \tilde{U}+f(x_1,y,\tilde{U}(x_1,y)e^{-\frac{c}{2}x_1})e^{\frac{c}{2}x_1}-\dfrac{c^2}{4}\tilde{U}=0 &\textrm{$x\in\O$}\\
   \tilde{U}=0 & \textrm{$x\in\partial\O$}\\
  \textrm{$\tilde{U}>0$ in $\O$}\\
  \textrm{$\tilde{U}(x_1,y)e^{-\frac{c}{2}x_1}$ is bounded.} \label{4.1.2}
\end{array} \right.
\end{equation}
The argument of Theorem \ref{thm:T1} for the nonexistence result can be applied if one can prove that $\tilde{U}$ decays exponentially as $|x_1|\to\infty.$ Moreover, if $\tilde{U}$ decays exponentially, the uniqueness can be obtained 
by using  variational argument as Theorem 2.3 in \cite{BR1}. Note that the sliding argument as of Theorem \ref{thm:T1} does not work in this case due to the lack of the Hopf lemma under Dirichlet condition. We end the proof by showing that $\tilde{U}$ really decays exponentially. 

From assumption (\ref{4.1.0}), we know that for any $\delta>0$ there exists $R=R(\delta)>0$ such that $f_s(x_1,y,0)\leq\mu(y)+\delta$ when $|x_1|\geq R$. Since $\omega$ is bounded, $\lambda_D(-\Delta_y-\mu(y),\omega)$ coincides with the classical Dirichlet principal eigenvalue in $\omega$, says $\lambda_0$. There exists an eigenfunction  $\phi\in L^\infty(\omega)$,  associated with $\lambda$, positive in $\omega$, such that 
$$\Delta_y\phi+\mu(y)\phi+\lambda_0\phi=0\quad\textrm{in $\omega$}\quad\quad;\quad\quad\phi=0\quad\textrm{on $\partial\omega$}.$$
One can actually find a function $\tilde{\phi}>0$ in $\overline{\omega}$ such that $\Delta_y\tilde{\phi}+(\tilde{\mu}(y)+\lambda_0)\tilde{\phi}=0$ in $\omega$, with $\tilde{\mu}(y)$ sufficiently close to $\mu(y)$ and  $\sup_{\overline{\omega}}|\tilde{\mu}-\mu|<\tilde{\delta}$, which is chosen later. Set $\mathcal{L}+\delta=\Delta+\mu(y)-c^2/4+\delta$,  conditions (\ref{5}) and (\ref{4.1.0}) yield that for any $\delta>0$ there exists $R=R(\delta)>0$ such that $(\mathcal{L}+\delta) \tilde{U}\geq0$ in $\O\setminus\O_R$. On the other hand, we set $w(x)=C\theta_a(x_1)\tilde{\phi}(y)$, where $\theta_a$ is the solution of equation
\begin{equation}
\left\{\begin{array}{ll}
 \theta''_{a}=(\kappa+\delta)\theta_a &\textrm{in $(R,R+a)$}\\
 \theta_{a}(R)=Ce^{\sqrt{\kappa}R}\\
 \theta_{a}(R+a)=Ce^{\sqrt{\kappa}(R+a)},\nonumber
\end{array} \right.
\end{equation}
with $\kappa=\lambda/2+c^2/4-2\delta$ ($\kappa>c^2/4$ if $\delta<\lambda_0/4$) and $C=\sup_{\O}\tilde{U}(x)e^{-\frac{c}{2}x_1}/\inf_{\overline{\omega}}\tilde{\phi}$. Choosing $\tilde{\delta}<\lambda/2$, direct computation yields $(\mathcal{L}+\delta)w\leq 0$ in $\overline{\O_{R+a}\setminus\O_R}$. The same argumentation of Proposition \ref{lem0} enables us to conclude that $U(x_1,y)\leq Ce^{-(\sqrt{\kappa}-c/2)|x_1|}\tilde{\phi}(y)$ in $\O$. This concludes the proof of theorem.
\end{proof}

By this preliminary, we are ready to present the main result of this section.

\subsection{Concentration of species in the more favorable region}
Based on the characterizations of the persistence and extinction of the species in the cylindrical domain $\O$ under Dirichlet boundary condition, we study the behavior of the species when a part of their habitat changes to be extremely unfavorable. More precisely, we consider  Eq. (\ref{4.1.1}) in the whole space $\R^N$ with two disjointed regions : the cylindrical domain $\O$ as in Subsection 4.1 and its complement $\O^c=\R^N\setminus \O$.  We shall prove that the species concentrate in the more favorable zone, $\O$, and the annihilation occurs in the dead zone $\O^c$ if the death rate in $\O^c$ becomes extremely high. Our goal is to characterize the limit of the sequence  $U_n(x)$, which are solutions of the equations 
\begin{equation}\label{add1.1}
\left\lbrace\begin{array}{ll}
\Delta U_n+c\partial_1 U_n+F_n(x,U_n)=0\quad x\in\R^N,\\
U_n>0 \textrm{ and bounded in $\R^N$}.
\end{array}
\right.
\end{equation}
For any $n$, the nonlinearities $F_n(x,s)$ are assumed to be continuous with respect to $x$ and of class $C^1$ with respect to $s$, $F_n(x,0)=0$, $\forall x\in\R^N$. Moreover,  $F_n(x,s)$ are assumed to satisfy 
\begin{equation}
\textrm{$\exists S>0$ such that $F_n(x,s)\leq 0$ for $s\geq S$,  $\forall x\in \R^N$,}\label{add1.0}
\end{equation}
\begin{eqnarray}
\begin{array}{cc}
&\textrm{$s\rightarrow F_n(x,s)/s $ is nonincreasing a.e in $ \R^N$ and there exist $D\subset\R^N$, $|D|>0$}\\
&\textrm{such that it is strictly decreasing in $D$}.\label{add1.0.0}
\end{array}
\end{eqnarray}
Let $f(x,s):\O\times [0,+\infty)\mapsto\R$ satisfy (\ref{4}), (\ref{5}), (\ref{4.1.0}), $\rho_n(x)=\frac{\partial F_n}{\partial s}(x,0)$, we assume further that:
\begin{equation}
\left\lbrace\begin{array}{ll}
F_n(x,s)=f(x,s)&\textrm{$x\in\O$, $s\in\R^+$ for all $n\in\mathbb{N}$}\\
\textrm{$F_n(x,s)$ and $\rho_n(x)$ are nonincreasing in $n$}&\textrm{$\forall(x,s)\in\overline{\O}\times\R^+$}\\
\textrm{$\rho_n(x)\to-\infty$ as $n\to\infty$}&\textrm{locally uniformly in $\O^c$.}
\label{add1.2}
\end{array}\right.
\end{equation}
%\mathop{\limsup}\limits_{x\in\O^c,|x|\to\infty}\rho_n(x)<\min\{0,\inf_{\overline{\omega}}\mu(y)\}&\forall n\in\mathbb{N}.
Before stating the result, let us briefly explain the meaning of this condition and of our achievement. This condition means that the environment of the species outside $\O$ is unfavorable and it becomes extremely unfavorable as $n\to\infty$. Our result confirms that no species can  persist outside $\O$ under such condition as $n\to\infty$. As is proved in Subsection 4.1, the species  is persistent in $\O$ if and only if $0\leq c<c^*$.  In the following result, we will see that as $n\to\infty$ the species can only persist in $\O$ and it is immediately mortal outside $\O$. Moreover, we will prove that the limit as $n\to\infty$ coincides with the unique solution of Eq. (\ref{4.1.1}) in $\O$ and zero in $\O^c$.

We derive the following result :
\begin{theorem}\label{add1} Let $U_n(x)$ be the sequence  of traveling front solution of Eq. (\ref{add1.1}) with $F_n(x,s)$ satisfies (\ref{add1.0})-(\ref{add1.2}).  Assume that  $c<c^*$ with $c^*$ is given in Theorem (\ref{add1_0}), then the following limit holds 
$$U_n(x)\to U_\infty(x)\quad\textrm{as $n\to\infty$},$$
uniformly for $x\in\R^N$, where $U_\infty\in W^{2,N}(\R^N)$ is nonnegative, vanishing in $\O^c$ and coincides with the unique positive solution of the following equation
\begin{equation}
\left\lbrace
\begin{array}{ll}
\Delta U+c\partial_1 U+f(x,U)=0&x\in\O\\
U(x)=0&x\in\partial\O.
\end{array}\right.\label{add1.4}
\end{equation} 
\end{theorem}
\begin{proof} Some arguments in the proof are inspired from \cite{GH}. 

Let us call $\mathcal{L}_n=\Delta+c\partial_1+\rho_n(x)$, defined in $\R^N$, and the generalized principal eigenvalue of $\mathcal{L}_n$ as following
 \begin{eqnarray}
\lambda_n=\sup\{\lambda\in\R:\exists\phi\in W^{2,N}_{loc}(\R^N),\phi>0,(\mathcal{L}_n+\lambda)\phi\leq0\textrm{ a.e in $\R^N$}\}.\nonumber
\end{eqnarray}
We also denote by $\lambda_D$ the generalized Dirichlet  principal eigenvalue of the operator $\Delta+c\partial_1+f_s(x,0)$ in $\O$ defined by (\ref{add1.3}). We have the following lemma:\\

\begin{lemma}\label{lem_add1} There holds that $\lambda_n$ converges  increasingly to $\lambda_D$. \\
\end{lemma}

Let us postpone the proof of this lemma for a moment to continue the proof of theorem. As is known, $c<c^*$ if and only if $\lambda_D<0$. By Lemma \ref{lem_add1} and assumption, we have $\lambda_n<\lambda_D<0$. Then Theorem 1.1 \cite{BR1} yields that the equation
\begin{equation}
\left\{\begin{array}{ll}
\Delta U_n+c\partial_1U_n+F_n(x,U_n)=0&x\in\R^N\\
0\leq U_n\leq {S}&x\in\R^N.
\end{array}\right.\label{add1.5}
\end{equation}
admits a strictly positive solution $U_n$. Let $V_n(x_1,y)=U_n(x_1,y)e^{\frac{c}{2}x_1}$, then $U_n$ is a solution of Eq. (\ref{add1.5}) if and only if $V_n$ is a solution of 
\begin{equation}
\left\{\begin{array}{ll}
  \Delta V_n+F_n(x_1,y,V_n(x_1,y)e^{-\frac{c}{2}x_1})e^{\frac{c}{2}x_1}-\dfrac{c^2}{4}V_n=0 &\textrm{$x\in\R^N$}\\
  \textrm{$V_n(x_1,y)e^{-\frac{c}{2}x_1}$ is bounded.} \label{add1.6}
\end{array} \right.
\end{equation}
For large $n\in\mathbb{N}$, one has $\mathop{\limsup}\limits_{|x|\in\O^c,|x|\to\infty}\rho_n(x)<0$,  Proposition 4, \cite{BR1} yields that $V_n(x)$ decays exponentially for $x\in\R^N\setminus\O$ and since  $F_n(x,s)=f(x,s)$ satisfies  condition (\ref{4.1.0}) in $\O$, Theorem \ref{add1_0} yields that $V_n(x)$ also decays exponentially for $x\in\O$. Thanks to condition (\ref{add1.0.0}), we derive, by  Theorem 1.1 of \cite{BR1}, that $U_n(x)$ is unique. Moreover, since $F_n$ is nonincreasing, for $m,k\in\mathbb{N}$, $k\geq m$, one has 
$$\Delta U_m+c\partial_1 U_m+F_k(x,U_m)=-F_m(x,U_m)+F_k(x,U_m)\leq0,\quad x\in\R^N.$$
Thus $U_m$ is a supersolution of equation satisfied by $U_k$. One can apply the comparison principle, Theorem 2.3 \cite{BR1}, to imply that $U_k\leq U_m$ in $\R^N$ for $k\geq m$. Then $U_n$ is nonincreasing with respect to $n$ and converges pointwise to a nonnegative function $U_\infty\leq {S}$. We will prove now that $U_\infty=0$ in ${\O}^c.$

From above arguments, $V_n(x)$ decays exponentially as $|x|\to\infty$. Multiplying $V_n(x)$ to Eq. (\ref{add1.6}), we derive, by applying the Stokes formula, that
\begin{eqnarray}
\int_{\R^N}\nabla V_n\cdot\nabla V_n=\int_{\R^N}F_n(x_1,y,V_ne^{-\frac{c}{2}x_1})e^{\frac{c}{2}x_1}V_n-\dfrac{c^2}{4}V_n^2\leq \int_{\R^N}F_0(x_1,y,V_ne^{-\frac{c}{2}x_1})e^{\frac{c}{2}x_1}V_n\leq\nonumber\\
\leq \mathop{\max}\limits_{x\in\R^N}\partial_s F_0(x,0)\int_{\R^N}V_0^2(x)\leq  M<\infty.\nonumber
\end{eqnarray}
This implies, by Lesbesgue monotone convergence theorem, that the sequence $V_n$ converges monotonically to some $V_\infty\in H^1(\R^N)$ as $n\to\infty$, weakly in $H^1(\R^N)$ and strongly in $L^2(\R^N)$. Moreover, taking an arbitrary compact set $K\subset\O^c$, one gets
\begin{eqnarray}
-(\mathop{\max}\limits_{K}\rho_n)\int_{K}V_n^2\leq-\int_K\rho_nV^2_n\leq-\int_K F_n(x,V_ne^{-\frac{c}{2}x_1})e^{\frac{c}{2}x_1}V_n=\nonumber\\
-\int_{\R^N}|\nabla V_n|^2-\int_{\R^N}\frac{c^2}{4}V_n+\int_{\R^N\setminus K}F_n(x,V_ne^{-\frac{c}{2}x_1})e^{\frac{c}{2}x_1}V_n\leq \int_{\R^N\setminus K}F_n(x,V_ne^{-\frac{c}{2}x_1})e^{\frac{c}{2}x_1}V_n\leq M\nonumber
\end{eqnarray}
Then, from assumption (\ref{add1.2}), we have $\mathop{\max}\limits_{K}\rho_n\to-\infty$, $\forall K\subset \O^c$, whence $V_\infty=0$ for all compact set in $\O^c$. This implies $V_\infty=0$ a.e in $\O^c$ or in the other words $U_\infty=0$ a.e in $\O^c$. As a consequence, the restriction of $U_\infty$ in $\O$ belongs to $H^1_{0}(\O)$. Moreover, since $F_n(x,s)=f(x,s)$ in $\O$, we have 
$$\Delta U_n+c\partial_1 U_n+f(x,U_n)=0\quad\quad x\in\O.$$
The standard elliptic estimates yield that $U_n\to U_\infty$ as $n\to\infty$ locally uniformly in $\O$ and moreover $U_\infty$ is a solution of the same equation in $\O$ in the weak $H^1_0(\O)$ sense. Thanks to Theorem \ref{add1_0}, we know that $U_\infty$ is unique and $U_\infty(x)\to0$ as $|x_1|\to\infty$, uniformly in $y\in\omega$. This convergence is actually uniform in $\R^N$ due to the nonincreasing monotonicity of $U_n$ with respect to $n$.

% Assume by contradiction that the convergence $U_n\to U_\infty$ as $n\to\infty$ is not uniform, one finds a positive constant $\varepsilon>0$ and a sequence $(x_k)\in\R^N$ such that for some $n_0\in\mathbb{N}$,
%$$|U_n(x_k)-U_\infty(x_k)|\geq\varepsilon,\quad\quad\forall n\geq n_0.$$
%Since the convergence is already locally uniform in $\R^N$, we deduce that $|x_n|\to\infty$ as $n\to\infty$. Since $U_n$ is nonincreasing with respect to $n$, we deduce that for some $0<\varepsilon_1<\varepsilon$, one has
%$$U_{n_0}(x_k)\geq\varepsilon_1+U_\infty(x_k).$$
%This is a contradiction since we know that $U_\infty(x), U_{n_0}(x)\to0$ as $|x|\to\infty$.\\

Lastly, to conclude the proof, it remains to prove Lemma \ref{lem_add1}, $\lambda_n\nearrow \lambda_D$. To this end, we first show that $\lambda_n< \lambda_D$, $\forall n\in\N$. Since $\rho_n(x)$ is nonincreasing in $n$, one sees that $\lambda_n$ is nondecreasing in $n$. Assume by contradiction that $\lambda_n\geq\lambda_D$ for some $n$. Let us denote by $\varphi_n$ and $\varphi$ respectively the principal eigenfunctions associated with $\lambda_n$ and $\lambda_D$, it holds that :
$$\Delta\varphi_n+c\partial_1\varphi_n+f_s(x,0)\varphi_n=-\lambda_n\varphi_n\leq-\lambda_D\varphi_n,\quad x\in\Omega,$$
and $\varphi_n>0$ in $\overline{\O}$. Note that the existence of a positive eigenfunction associated with the generalized principal eigenvalue $\lambda_n$ in unbounded domain is given in \cite{BR3}. Because $\varphi_n$ is a supersolution of equation satisfied by $\varphi$, if there exists $0<\kappa<+\infty$ such that $\kappa\varphi\leq\varphi_n$ in $\O$, one can enlarge $\kappa$ until $\kappa\varphi$ touches $\varphi_n$ from below at some point. The strong maximum principle implies $\kappa\varphi\equiv\varphi_n$ in $\O$, which is impossible because $\varphi=0$ on $\partial\O$. In other words, $\sup\{\kappa\in(0,+\infty],\kappa\varphi\leq\varphi_n \textrm{ in $\O$}\}=+\infty$. This yields another contradiction since $\varphi>0$ in $\O$. As a result, $\lambda_n<\lambda_D$, $\forall n\in\N$. Next, we aim to show the limit $\lim_{n\to\infty}\lambda_n=\lambda_D$. Since $\lambda_n$ is nondecreasing and bounded from above, there exists  ${\lambda}_\infty=\lim_{n\to\infty}\lambda_n\leq\lambda_D$. We shall prove that ${\lambda}_\infty=\lambda_D$.

Observe that, by the transformation $\tilde{\varphi}_n=\varphi_ne^{\frac{c}{2}x_1}$, we see that $\tilde{\varphi}_n$ satisfies the equation
\begin{equation}
\Delta\tilde{\varphi}_n+\rho_n(x)\tilde{\varphi}_n-\frac{c^2}{4}\tilde{\varphi}_n+\lambda_n\tilde{\varphi}_n=0,\quad\quad x\in\R^N.\label{add1.6.1}
\end{equation}
Let us show that $\tilde{\varphi}_n$ decays exponentially. Indeed, as in the proof of Theorem \ref{add1_0}, let $\tilde{\phi}$ be the function such that $\inf_{\overline{\omega}}\tilde{\phi}>0$ and $\Delta_y\tilde{\phi}+\tilde{\mu}(y)\tilde{\phi}=0$ in $\omega$, with $\tilde{\mu}(y)$ sufficiently close to $\mu(y)-\lambda_0$, where  $\lambda_0=\lambda_D(-\Delta_y-\mu(y),\omega)>0$. Then, one sees that, $(\lambda_n,\tilde{\varphi}_n)$ is the principal eigenpair of the operator $\mathcal{\tilde{L}}_n=\Delta+\rho_n(x)-c^2/4$ if and only if $(\lambda_n,\phi_n)$, with $\phi_n=\tilde{\varphi}_n/\tilde{\phi}$, is the principal eigenpair of the following operator
$$\Delta+\frac{2\nabla\tilde{\phi}\cdot\nabla}{\tilde{\phi}}+\rho_n(x)-{\mu}(y)-\lambda_0-\frac{c^2}{4}.$$
By  assumptions (\ref{4.1.0}) and (\ref{add1.2}), one has 
\begin{equation}
\limsup_{|x|\to\infty}\left\lbrace\rho_n(x)-{\mu}(y)-\lambda_0-\frac{c^2}{4}\right\rbrace<0<-\lambda_n.\label{add1.6.21}
\end{equation}
Hence, applying Proposition 1.11 \cite{BR3}, we know that $\lambda_n$ is simple, moreover $\phi_n$ is unique (up to multiplications) and decays exponentially as $|x|\to\infty$. It follows immediately that $\tilde{\varphi}_n$ also decays exponentially.

On the other hand, since $\lambda_n<\lambda_D<0$ and $\rho_n(x)$ is nonincreasing in $n$, one has 
$$\Delta\tilde{\varphi}_n+\rho_0(x)\tilde{\varphi}_n\geq0.$$
From above, $\tilde{\varphi}_n$ is bounded and $\tilde{\varphi}_n$ solves  linear equation (\ref{add1.6.1}), we can normalize  $\tilde{\varphi}_n$ in such the way $\sup_{\R^N}\tilde{\varphi}_n=1$. Thanks to   conditions (\ref{4.1.0}) and $\limsup_{x\in\O^c,|x|\to\infty}\rho_0(x)<0$, the same argumentation of Proposition 8.6 in \cite{BR3} may be applied to derive that there exists an exponential decay function $\overline{\varphi}$  depending only on  $\rho_0(x)$ such that $\tilde{\varphi}_n\leq \overline{\varphi}$ in  $\R^N$. From the equation (\ref{add1.6.1}), one has
\begin{eqnarray}
\int_{\R^N}{|\nabla\tilde{\varphi}_n|}^2\leq \int_{\R^N}\rho_n(x)\tilde{\varphi}_n^2(x)+\lambda_n\int_{\R^N}\tilde{\varphi}_n^2(x)\leq \int_{\R^N}\rho_0(x)\overline{\varphi}^2(x)<\infty\label{add1.6.2}
\end{eqnarray}
This implies that there exists $\varphi_\infty\in H^1(\R^N)$, with $\sup_{\R^N}\varphi_\infty=1$, such that $\tilde{\varphi}_n$ converges up to subsequence to $\varphi_\infty$  weakly in  $H^1(\R^N)$ and strongly in $L^2(K)$ for all compact set $K\subset\R^N$. For any compact set $K\subset\O^c$, we derive from (\ref{add1.6.2})
$$-\max_{K}\rho_n\int_K\tilde{\varphi}_n^2dx\leq-\int_K\rho_n\tilde{\varphi}_n^2dx\leq-\int_{\R^N}{|\nabla\tilde{\varphi}_n|}^2+\int_{\R^N\setminus K}\rho_n\tilde{\varphi}_n^2dx\leq \sup_{\R^N}\rho_0\int_{\R^N}\ol\varphi^2dx.$$
Since, from (\ref{add1.2}) for all $K\subset\O^c$, $-\max_{K}\rho_n\to\infty$ as $n\to\infty$, we have $\varphi_\infty=0$ a.e in $K$ and then a.e in $\O^c$. Lastly, again from (\ref{add1.6.1}), one has
\begin{eqnarray}
\int_{\R^N}{|\nabla\tilde{\varphi}_n|}^2\leq \int_{\R^N}\rho_n(x)\tilde{\varphi}_n^2(x)+\lambda_n\tilde{\varphi}_n^2(x)-\frac{c^2}{4}\tilde{\varphi}_n^2(x)\leq \int_{\R^N}\rho_0(x)\tilde{\varphi}_n^2(x)dx+{\lambda}_\infty\tilde{\varphi}_n^2(x)-\frac{c^2}{4}\tilde{\varphi}_n^2(x).\nonumber
\end{eqnarray}
Since $\tilde{\varphi}_n\leq\bar{\varphi}$,   we derive, by Lebesgue dominated convergence theorem that
$$\int_{\R^N}\rho_0\tilde{\varphi}_n^2(x)dx\to\int_{\R^N}\rho_0{\varphi}_\infty^2(x)dx=\int_{\O}f_s(x,0){\varphi}_\infty^2(x)dx$$
$${\lambda}_\infty\int_{\R^N}\tilde{\varphi}_n^2(x)dx\to{\lambda}_\infty\int_{\O}{\varphi}_\infty^2(x)dx\quad\textrm{and}\quad\frac{c^2}{4}\int_{\R^N}\tilde{\varphi}_n^2(x)dx\to\frac{c^2}{4}\int_{\O}{\varphi}_\infty^2(x)dx$$
Whence, the lower semicontinuity property  yields
\begin{eqnarray}
\int_{\O}{|\nabla{\varphi}_\infty|}^2=\int_{\R^N}{|\nabla{\varphi}_\infty|}^2\leq \liminf_{n\to\infty}\int_{\R^N}{|\nabla\tilde{\varphi}_n|}^2\leq \int_{\O}f_s(x,0){\varphi}_\infty^2(x)+{\lambda}_\infty{\varphi}_\infty^2(x)-\frac{c^2}{4}{\varphi}_\infty^2(x)\label{0.0}
\end{eqnarray}
By the Liouville transformation, $\lambda_D$ is the principal eigenvalue of a self-adjoint operator. We know from \cite{BR3} that it has a variational structure. 
$$\lambda_D=\inf_{w\in C_c^1(\O),w\not\equiv0}\frac{\int_{\O}|\nabla w|^2-f_s(x,0)w^2+\frac{c^2}{4}w^2dx}{\int_{\O}w^2dx}.$$
Since $C_c^1(\O)$ is dense in $H^1_0(\O)$ and $\varphi_\infty\in H^1_0(\O)$, there exists a sequence $w_n\in C_c^1(\O)$  converges to $\varphi_\infty$ strongly in $L^2(\O)$ and $H^1(\O)$. Combining with (\ref{0.0}), we derive
\begin{eqnarray}
\lambda_D\leq \frac{\int_{\O}|\nabla w_n|^2-f_s(x,0)w_n^2+\frac{c^2}{4}w_n^2dx}{\int_{\O}w_n^2dx}\to \frac{\int_{\O}|\nabla\varphi_\infty|^2-f_s(x,0)\varphi_\infty^2+\frac{c^2}{4}\varphi_\infty^2dx}{\int_{\O}\varphi_\infty^2dx}\leq\lambda_\infty.\label{add1.6.22}
\end{eqnarray}
Eventually, we obtain $\lambda_D=\lambda_\infty$. This completes the proof.
\end{proof}
\begin{remark} In the proof, we have proved a result, which is stronger than what we really need. In fact, to obtain the conclusion of Theorem \ref{add1}, one only needs to prove $\lambda_n<\lambda_D$. However, by proving Lemma \ref{lem_add1}, we obtain a more interesting result on the convergence of the eigenvalues and eigenfunctions. This indeed makes  Theorem \ref{add1} more transparent and more interesting.
\end{remark}
%\begin{remark} Inequality (\ref{add1.6.22}) in fact implies that 
%$$\lambda_D=\inf_{w\in H_0^1(\O),\|w\|_{L^2(\O)}=1}\int_{\O}|\nabla w|^2-f_s(x,0)w^2dx+\frac{c^2}{4}.$$
%By assumption (\ref{4.1.0}) and the same arguments as (\ref{add1.6.1})-(\ref{add1.6.21}), we can actually prove that $\lambda_D$ is simple. This proof is referred to Proposition 1.11 \cite{BR3}. In this case, due to the density of $C_c^1(\O)$ in $H^1_0(\O)$, the infima of variational characterization of $\lambda_D$ taken over $H^1_0(\O)$ and $C_c^1(\O)$ are equivalent and equal to the unique $\lambda_D$. This fact is not true in general, especially when (\ref{4.1.0}) does not hold. Here, thanks to linear structure  of Eq. (\ref{add1.6.1}), we can normalize $\tilde{\varphi}_n$ of the sup-norm equal $1$. Therefore, we can find a uniform exponential decay, $\ol\varphi\geq \tilde{\varphi}_n$, depending only on $\rho_0(x)$, which helps us to overcome the lack of compactness of $\O$ and $\R^N$.
%\end{remark}

The last result is concerned with a further    qualitative property of the  fronts of Eq. (\ref{3}), namely the symmetry breaking in $x_1$ axis. To this aim, the monotonicity and exact asymptotic behavior of the fronts play the crucial role. From  Proposition \ref{lem0}, we know that the fronts $U(x_1,y)$ decay exponentially as $x\to\pm\infty$. Therefore, natural questions may arise, which are the right conditions such that the fronts are monotone when $|x_1|$ large enough and whether they are symmetric in $x_1$ axis.  These questions are addressed in the following by studying the asymptotic behavior of solutions as $x_1\to\pm\infty$.

\subsection{Symmetry breaking of the fronts}

\begin{theorem}\label{thm:add2}
 Let $U$ be a traveling front solution of Eq. (\ref{3}) with $f$ is such that
\begin{equation}
\left\lbrace\begin{array}{cc}
|f_s(x_1,y,0)-\alpha(y)|=O(e^{px_1})\quad\textrm{as $x_1\to-\infty$, and} \quad\lambda_\alpha=\lambda_N(-\Delta_y-\alpha(y),\omega)>0\\
|f_s(x_1,y,0)-\beta(y)|=O(e^{-qx_1})\quad\textrm{as $x_1\to+\infty$, and} \quad\lambda_\beta=\lambda_N(-\Delta_y-\beta(y),\omega)>0
\end{array}\right.\label{add2.1}
\end{equation}
uniformly in $y\in\omega$, for some $\alpha,\beta\in\L^\infty(\omega)$, $p,q>0$. We assume further that  $s\to f(x,s)\in C^{1,r}(0,\delta)$ for some $r,\delta>0$. Then, $U$ is asymmetric if $$\lambda_\beta\neq\lambda_\alpha+c^2-2c\sqrt{\lambda_\alpha+\frac{c^2}{4}},$$ where $c$ is the given forced speed of traveling front.

\end{theorem}

\begin{proof} We investigate at first the precise asymptotic behavior of solution of Eq.(\ref{3}) on the branch $\O^-$, By analogy, we derive also the asymptotic behavior on  the branch $\O^+$. According to Proposition \ref{lem0} and Theorem 5.1, Appendix, for any $\delta>0$, we have shown that
\begin{equation}\label{add2.2}
C_{2,\delta}e^{\kappa_\delta x_1}\leq U(x_1,y)\leq C_{1,\delta}e^{\tau_\delta x_1}\quad\quad\forall (x_1,y)\in\O^-,
\end{equation}
where $\kappa_\delta=\sqrt{{\lambda}_\alpha+\delta+\dfrac{c^2}{4}}-\dfrac{c}{2}$ and  ${\tau}_\delta=\sqrt{{\lambda}_\alpha-\delta+\dfrac{c^2}{4}}-\dfrac{c}{2}$. Since $\omega$ is bounded, we refer to \cite{BNV}, that there exists a unique (up to a multiplication) eigenfunction $\varphi$  with Neumann boundary condition associated to $\lambda_\alpha$:
\begin{equation}
\left\{\begin{array}{ll}
-(\Delta_y+\alpha(y))\varphi={\lambda_\alpha}\varphi &\textrm{in $\omega$},\\
\partial_\nu\varphi(y)=0 &\textrm{on $\partial\omega$}.\nonumber
\end{array} \right.
\end{equation}
We rewrite Eq. (\ref{3}) as in the following form
\begin{equation}
\left\lbrace\begin{array}{ll}
\mathcal{M}U=-\Delta U-c\partial_1 U-\alpha(y) U=H_f(x_1,y)\\
H_f(x_1,y)=f(x_1,y,U)-\alpha(y) U\nonumber
\end{array}
\right.
\end{equation}
By the regularity condition $s\to f(x,s)\in C^{1,r}(0,\delta)$ and condition (\ref{add2.1}), we have
\begin{eqnarray}
|H_f(x_1,y)| &\leq& |f(x_1,y,U)-f_s(x_1,y,0)U|+|f_s(x_1,y,0)U-\alpha(y)U|\nonumber\\
&\leq& C_1 e^{(r+1)\tau_\delta x_1}+C_2e^{(p+\tau_\delta)x_1}\leq C_3e^{m(\tau_\delta)x_1}\quad\textrm{ as $x_1\to-\infty$},\label{add2.3}
\end{eqnarray}
where $m(\tau_\delta)=\min\{(r+1)\tau_\delta,p+\tau_\delta\}$. By Theorem 4.3 of \cite{BN}, we can write $U$ as follow
\begin{equation}
U=u^0(x_1,y)+u^*(x_1,y),\nonumber
\end{equation}
where $(u^0,u^*)$ is a solution of system 
\begin{equation}
\left\{\begin{array}{ll}
\mathcal{M} u^0=0 & x\in\O^-\\
\partial_\nu u^0=0 & \partial \O^-
\end{array} \right.\quad\quad \begin{array}{ll}
\mathcal{M}u^*=H_f(x_1,y)& x\in\O^-\\
\partial_\nu u^*=0 & \partial \O^-.\label{add2.3.0}
\end{array}
\end{equation}
Moreover, $u^0$ has a precisely exponential asymptotic behavior as $x_1\to-\infty$, namely there exist $\lambda>0$ and $\psi(x_1,y)=(-x_1)^k\psi_k(y)+...+\psi_0(y)\not\equiv0$ such that 
\begin{equation}
\left\{\begin{array}{ll}
u^0(x_1,y)=e^{\lambda x_1}\psi(x_1,y)+O(e^{\lambda x_1})\\
\nabla u^0(x_1,y)=\nabla(e^{\lambda x_1}\psi(x_1,y))+O(e^{\lambda x_1}),
\end{array} \right.\label{add2.3.1}
\end{equation}
and for any $\varepsilon>0$, $u^*$ satisfies the inequality  
\begin{equation}
|u^*(x_1,y)|+|\nabla u^*(x_1,y)|\leq C_{\varepsilon,\delta} e^{(m(\tau_\delta)-\varepsilon)x_1}\quad\quad\textrm{for some $C_\varepsilon>0$}. \label{add2.4}
\end{equation}
Let us define
$$\tau_0=\sup\{\tau:\exists C_\tau\textrm{ such that } u^0(x_1,y)\leq C_\tau e^{\tau x_1} \textrm{ in $\O^-$}\}.$$
Inequalities (\ref{add2.2}) yields $\kappa_\delta\leq\tau_0\leq \tau_\delta$, for any $\delta>0$, thus $\tau_0$ is indeed  a real number. We want to prove that $\tau_0=\sqrt{{\lambda}_\alpha+\dfrac{c^2}{4}}-\dfrac{c}{2}$. Taking $\tau<\tau_0$, then $0\leq u^0(x_1,y)\leq C_\tau e^{\tau x_1}$ and moreover  $|\nabla u^0(x_1,y)|\leq C_\tau' e^{\tau x_1} $ by the Harnack inequality.
Proceeding as (\ref{add2.3}), we get : $|H_f(x_1,y)|\leq C_4e^{m(\tau)x_1}$, where $m(\tau)=\min\{(r+1)\tau,p+\tau\}>\tau$. As a result of  (\ref{add2.4}), we have
\begin{equation}
|u^*(x_1,y)|+|\nabla u^*(x_1,y)|\leq D_\tau e^{\frac{(\tau+m(\tau))x_1}{2}}. \nonumber
\end{equation} 
One sees that as $\tau\nearrow \tau_0$,
$\dfrac{\tau+m(\tau)}{2}\nearrow\dfrac{\tau_0+m(\tau_0)}{2}>\tau_0$. Therefore, there exist $\epsilon>0$ and $C_\epsilon>0$ such that\begin{equation}
|u^*(x_1,y)|+|\nabla u^*(x_1,y)|\leq C_\epsilon e^{(\tau_0+\epsilon)x_1}\quad \textrm{in $\O^-$}.\label{add2.5}
\end{equation}
It follows immediately that $\forall\tau<\tau_0$ 
$$|u^0(x_1,y)|\leq |u^0(x_1,y)|+|u^*(x_1,y)|\leq C_\tau e^{\tau x_1}+C_\epsilon e^{(\tau_0+\epsilon)x_1}\leq (C_\tau+C_\epsilon)e^{\tau x_1}\quad \textrm{in $\O^-$}.$$
On the other hand, for $\delta$ small enough
$$u_0(x_1,y)=u^0(x_1,y)-u^*(x_1,y)\geq C_{2,\delta}e^{(\sqrt{{\lambda}_\alpha+\delta+\frac{c^2}{4}}-\frac{c}{2})x_1}-C_\epsilon e^{(\tau_0+\epsilon)x_1}\geq C_{3,\delta}e^{(\sqrt{{\lambda}_\alpha+\delta+\frac{c^2}{4}}-\frac{c}{2})x_1}.$$
Applying Theorem 4.2 of \cite{BN} to $u^0$, we deduce that there is exactly one positive constant $\lambda$ such that $\tau\leq\lambda\leq\sqrt{{\lambda}_\alpha+\delta+\frac{c^2}{4}}-\frac{c}{2}$, $\forall \tau<\tau_0$ and (\ref{add2.3.1}) holds for a suitable exponential solution $w(x_1,y)=e^{\lambda x_1}\psi(y)$. From (\ref{add2.5}), $\lambda$ cannot be strictly bigger than $\tau_0$, therefore we must have $\lambda=\tau_0>0$. Since $u^0>0$, we deduce $\psi_k>0$ and thus  Theorem 2.4 of \cite{BN} yields that  $\psi(y)$ is a solution of
\begin{equation}
\left\{\begin{array}{ll}
-(\Delta_y+\alpha(y))\psi=(\lambda^2+c\lambda)\psi&\textrm{in $\omega$}\\
\psi_\nu=0&\textrm{on $\partial\omega$}.
\end{array} \right.\label{add2.6}
\end{equation}
Since ${\lambda}_\alpha>0$, Theorem 2.1 of \cite{BN} implies that (\ref{add2.6}) possesses exactly one positive principal eigenvalue, that is  $\lambda={\dfrac{-c+\sqrt{c^2+4{\lambda}_\alpha}}{2}}=\tau_0$. We obtain the precisely asymptotic behavior of $U(x_1,y)$ as $x_1\to-\infty$. 

By analogy, we obtain  the precise exponential behavior of $U(x_1,y)$ as $x_1\to+\infty$. It is precisely exponentially asymptotic as $x_1\to+\infty$ with the exponent $\lambda'={\dfrac{-c-\sqrt{c^2+4{\lambda}_\beta}}{2}}.$ As a consequence, we have proved that
$$U(x_1,y)\sim C_1e^{-(\frac{c+\sqrt{c^2+4{\lambda}_\beta}}{2})x_1}\quad\textrm{as $x\to+\infty$};\quad U(x_1,y)\sim C_2e^{(\frac{-c+\sqrt{c^2+4{\lambda}_\alpha}}{2})x_1}\quad\textrm{as $x\to-\infty$},$$
uniformly in $y$. This result, in particular, implies that $U(x_1,y)$ is increasing in  $(-\infty,-R)\times\omega$ and decreasing in $(R_1,\infty)\times\omega$ for $R, R_1$ large enough. To achieve the symmetry in $x_1$, necessarily, we have
$$\frac{c+\sqrt{c^2+4{\lambda}_\beta}}{2}=\frac{-c+\sqrt{c^2+4{\lambda}_\alpha}}{2}\quad\Longleftrightarrow\quad\lambda_\beta=\lambda_\alpha+c^2-2c\sqrt{\lambda_\alpha+\dfrac{c^2}{4}}.$$
In other words $U$ is asymmetric if $\lambda_\beta\neq\lambda_\alpha+c^2-2c\sqrt{\lambda_\alpha+\dfrac{c^2}{4}}$.
\end{proof}
\begin{remark}
we see that  if $c\neq0$,  the asymmetry holds   when $\lambda_\beta=\lambda_\alpha$. The drift term is therefore the main inducement that makes the front asymmetric. However, we do not know that whether the front is symmetric  when 
$$\lambda_\beta=\lambda_\alpha+c^2-2c\sqrt{\lambda_\alpha+\dfrac{c^2}{4}}.$$
The answer of this question requires more involved analysis. We state this as an open question. Our result applies, in particular, to show that the asymmetry holds when the nonlinearity $f$ satisfying  (\ref{add2.1}) becomes as (\ref{0.00new}), namely $\lambda_\beta=\lambda_\alpha=m>0$ and
$$|f_s(x_1,y,0)+m|=O(e^{-p|x_1|})\quad\textrm{as $|x_1|\to\infty$, for some $m,p>0$}.$$

\end{remark}
\begin{remark} We point out that the  assumption on the exponential rate of convergences $f_s(x_1,y,0)\to\alpha(y)$ and $f_s(x_1,y,0)\to\beta(y)$ in (\ref{add2.1}) is important. Indeed, if (\ref{add2.1}) does not hold, the precise exponential behavior of  $u^0$ satisfying Eq. (\ref{add2.3.0})  may not be true as (\ref{add2.3.1}) in general. For instance, in one dimensional space, $\beta(y)\equiv-1/2$, if $f_s(x,0)$ converges slowly to $-1/2$, we can take $w(x)=xe^{- x}$, which is a solution of 
$$w''+\frac{1}{2}w'+g(x)w=0\quad\textrm{in $\R\setminus(2,-\infty)$},\quad\quad g(x)=-\frac{1}{2}+\frac{1}{2x}\to-\frac{1}{2}\quad\textrm{as   $x\to+\infty$}.$$

\end{remark}

\section{Appendix}
\label{appendix}\vspace{1cm}

\begin{theorem}\label{thm:A1} Let $U$ be a traveling front solution of (\ref{3}). Assume that (\ref{5}), (\ref{6}) hold and $f$ is such that
\begin{equation}\label{A1.2}
\liminf_{x_1\to\pm\infty}f_s(x_1,y,0)\geq\alpha_\pm(y)\quad\textrm{and}\quad\lambda_{\alpha_\pm}=\lambda_N(-\Delta_y-\alpha_\pm(y),\omega)>0,
\end{equation} 
for some functions $\alpha_\pm\in\L^\infty(\omega)$. Then, for any $\delta>0$, there exist $A_\pm>0$ and $\tau_{\alpha_\pm}\geq\sqrt{\lambda_{\alpha_\pm}+\delta+\dfrac{c^2}{4}}$ such that
$$U(x_1,y)\geq A_- e^{(\tau_{\alpha_-}-\frac{c}{2}) x_1}\quad\textrm{$\forall(x_1,y)\in\O^- $}\quad\textrm{and}\quad U(x_1,y)\geq A_+e^{-(\tau_{\alpha_+}+\frac{c}{2}) x_1}\quad\textrm{$\forall(x_1,y)\in\O^+ $}.$$
\end{theorem}
\begin{proof}
Since (\ref{5}), (\ref{6}) hold, from Proposition $(\ref{lem0})$, we know that $U(x_1,y)$ decays  exponentially as $|x_1|\to\infty$. Let us denote  $I=\{\alpha_-,\alpha_+\}$. We know from \cite{BNV} that there exist  principal eigenfunctions $\varphi_i$  associated with $\lambda_i$ such that for $i\in I$
\begin{equation}
\left\{\begin{array}{ll}
  -\Delta \varphi_i-i(y)\varphi_i=\lambda_i\varphi_i &\textrm{in $\omega$}\\
  \partial_\nu \varphi_i=0 & \textrm{on $\partial\omega$}.\nonumber
\end{array} \right.
\end{equation}
For $i\in I$, $\delta>0$, we set $\widetilde{\mathcal{L}}_i=\Delta_x+c\partial_1 +i(y)-\delta$. By assumption (\ref{A1.2}), there exists $R=R(\delta)>0$ such that $\widetilde{\mathcal{L}}_{\alpha_-} U\leq 0$ in $\O^-\setminus\O_{R}$ and $\widetilde{\mathcal{L}}_{\alpha_+} U\leq 0$ in $\O^+\setminus\O_{R}$. Define the functions $$\omega_{\alpha_-}(x)=e^{\tau_{\alpha_-} x_1}\varphi_{\alpha_-}(y) \quad\quad \textrm{and} \quad\quad  \omega_{\alpha_+}(x)=e^{-\tau_{\alpha_+} x_1}\varphi_{\alpha_+}(y),$$ 
direct computation yields
$$\textrm{$\widetilde{\mathcal{L}}_{\alpha_-}\omega_{\alpha_-}=\left( \tau_{\alpha_-}^2+c\tau_{\alpha_-}-\lambda_{\alpha_-}-\delta\right)\omega_{\alpha_-}\geq0$ in $\O^-\setminus\O_{R}$  if $\tau_{\alpha_-}\geq\frac{\sqrt{c^2+4(\lambda_{\alpha_-}+\delta)}-c}{2}$};$$
$$\textrm{$\widetilde{\mathcal{L}}_{\alpha_-}\omega_{\alpha_+}=\left( \tau_{\alpha_+}^2-c\tau_{\alpha_+}-\lambda_{\alpha_+}-\delta\right)\omega_{\alpha_+}\geq0$ in $\O^+\setminus\O_{R}$ if $\tau_{\alpha_+}\geq\frac{\sqrt{c^2+4(\lambda_{\alpha_+}+\delta)}+c}{2}$}.$$
The strong maximum principle and the Hopf lemma yield $\inf_{y\in\omega}\varphi_i(y)>0$, $\inf_{y\in\omega}U(-R,y)>0$ and $\inf_{y\in\omega}U(R,y)>0$. Therefore, we can choose the positive constants $C_{\delta,i}$ small enough such that the functions satisfy $W_{\alpha_-}(x_1,y)=U(x_1,y)-C_{\delta,{\alpha_-}}\omega_{\alpha_-}(x_1,y)\geq 0$  and $W_{\alpha_+}(x_1,y)=U(x_1,y)-C_{\delta,{\alpha_+}}\omega_{\alpha_+}(x_1,y)\geq 0$ for $x_1=-R,y\in\omega$. They have the Neumann boundary conditions $\partial_\nu W_i=0$ on $\partial\O$ and satisfy the  inequalities $\widetilde{\mathcal{L}}_i W_i\leq 0$ in $\O^\pm\setminus\O_R$. Set $z_i(x_1,y)=W_i(x_1,y)/\varphi_i(y)$, we get $\partial_\nu z_i=0$ on $\partial \O$ and
$$0\geq\dfrac{\widetilde{\mathcal{L}}_{\alpha_-} W_{\alpha_-}}{\varphi_{\alpha_-}}=\Delta z_{\alpha_-}+c\partial_1 z_{\alpha_-}+2\dfrac{\nabla_y\varphi_{\alpha_-}}{\varphi_{\alpha_-}} .\nabla_y z_{\alpha_-}-\left(\lambda_{\alpha_-}+\delta\right) z_{\alpha_-}\quad x_1<-R,\forall y\in\omega,$$
$$0\geq\dfrac{\widetilde{\mathcal{L}}_{\alpha_+} W_{\alpha_+}}{\varphi_{\alpha_+}}=\Delta z_{\alpha_+}+c\partial z_{\alpha_+}+2\dfrac{\nabla_y\varphi_{\alpha_+}}{\varphi_{\alpha_+}} .\nabla_y z_{\alpha_+}-\left(\lambda_{\alpha_+}+\delta\right) z_{\alpha_+},\quad x_1>R,\forall y\in\omega.$$
Since $z_{\alpha_-}(x_1,y)\to0$ as $x_1\to-\infty$, $z_{\alpha_+}(x_1,y)\to0$ as $x_1\to+\infty$, and zero-order coefficients of elliptic-operators with respect to $z_i$ are negative, the weak maximum principle is applied to derive $z_i\geq 0$ in $\O^\pm\setminus\O_R$. As a consequence, there exist $\tau_i\geq\sqrt{\lambda_i+\delta+\dfrac{c^2}{4}}$ such that 
\begin{equation}\label{A1.3}
\left\lbrace\begin{array}{ll}
 C_{\delta,{\alpha_+}} e^{-(\tau_{\alpha_+}+\frac{c}{2}) x_1}\varphi_{\alpha_+}(y) \leq U(x_1,y)\quad\textrm{in $\O^+\setminus\O_{R}$.}
\\
 C_{\delta,{\alpha_-}} e^{(\tau_{\alpha_-}-\frac{c}{2}) x_1}\varphi_{\alpha_-}(y) \leq U(x_1,y)\quad\textrm{in $\O^-\setminus\O_{R}$.}
\end{array}\right.
\end{equation}
By the Harnack inequality, one has $\inf_{|x_1|\leq R}U(x_1,y)>0$, we deduce that $C_{\delta,i}$ indeed can be chosen such that the inequalities at (\ref{A1.3}) hold respectively in $\O^\pm$. The proof is complete.
\end{proof}

\begin{theorem}\label{thm:A2}
Let $U\in W^{1,2}_{N+1,loc}(\R\times\R^{N})$ be a solution of (\ref{12}), where $f$ is such that  conditions (\ref{10}),(\ref{11}) hold. Then there exist two positive constants $k$ and $\varepsilon$ such that $$\forall(t,x)\in\R\times\R^N\quad\quad U(t,x_1,y)\leq ke^{-\varepsilon|x_1|}.$$
\end{theorem}
\begin{proof}
We only need to prove that the statement holds for $x_1\geq0$ and by analogy we also derive the result for $x_1\leq0$. Using the transformation $V(t,x_1,y)=U(t,x_1,y)e^{\frac{c}{2}x_1}$, we see that $V(t,x)$ is  periodic in $y$, T-periodic in $t$ and satisfies the following equation
\begin{equation}
\left\{\begin{array}{ll}
  V_t=\Delta V+f(t,x,Ve^{-\frac{c}{2}x_1})e^{\frac{c}{2}x_1}-\dfrac{c^2}{4}V &\textrm{$t\in\R,x\in\R^N$}\\
 \textrm{$Ve^{-\frac{c}{2}x_1}$ is bounded.}
\end{array} \right.\nonumber
\end{equation}
For any $R>0$, $\delta>0$, we denote $\mathcal{Q}_R=\R\times[0,R]\times\R^{N-1}$ and set $\mathcal{L}_\delta=\partial_t-\Delta-\gamma(t,y)-\delta+\dfrac{c^2}{4}$. By condition (\ref{11}), there exists $R=R(\delta)>0$ such that $\mathcal{L}_\delta V\leq 0$,  $\forall (t,x)\in\mathcal{Q}_R$. Moreover, there exists a unique  eigenpair $(\lambda,\varphi)$ satisfying (\ref{T4.0}). Fix $\tau\in\R$ and define the function $\upsilon(t,x)=\theta_{a}(x_1)\varphi(t,y)e^{(\tau-t)\delta}$, where $\theta_{a}: [R,R+a]\to\R$ is the solution of  
\begin{equation}
\left\{\begin{array}{ll}
 \theta''_{a}=(\kappa+\delta)\theta_a &\textrm{in $(R,R+a)$}\\
 \theta_{a}(R)=Ce^{\sqrt{\kappa}R}\\
 \theta_{a}(R+a)=Ce^{\sqrt{\kappa}(R+a)},\nonumber
\end{array} \right.
\end{equation}
where $C=\sup_{\R\times\R^{N}}V(t,x)e^{-\frac{c}{2}x_1}/\inf_{\R\times\R^{N-1}}\varphi(t,y)$, $\kappa>0$ would be chosen later. Note that $C\in(0,\infty)$ since $\inf_{\R\times\R^{N-1}}\varphi(t,y)>0$ due to the periodicity of $\varphi$ in $y$ and $t$. Direct calculation yields
\begin{eqnarray}
\theta_a(\rho)=C(e^{(\sqrt{\kappa}+\sqrt{\kappa+\delta})R})\left(1-\dfrac{e^{\sqrt{\kappa}a}-e^{-\sqrt{\kappa+\delta}a}}{e^{\sqrt{\kappa+\delta}a}-e^{-\sqrt{\kappa+\delta}a}}\right)e^{-\sqrt{\kappa+\delta}\rho}+\nonumber\\
+C(e^{(\sqrt{\kappa}-\sqrt{\kappa+\delta})R})\dfrac{-e^{-\sqrt{\kappa+\delta}a}+e^{\sqrt{\kappa}a}}{e^{\sqrt{\kappa+\delta}a}-e^{-\sqrt{\kappa+\delta}a}}e^{\sqrt{\kappa+\delta}\rho}.\nonumber
\end{eqnarray}
Choosing $0<\delta<\lambda/3$ and $\kappa\in\left({c^2}/{4},\lambda+{c^2}/{4}-3\delta\right)$, we have $\mathcal{L}_\delta\upsilon=\left(-\kappa-3\delta+\lambda+{c^2}/{4}\right)\upsilon\geq0$. Moreover, the way of choosing $C$ yields $\upsilon(t,x_1,y)\geq V(t,x_1,y)$ for $t\leq\tau,x_1\in\{R,R+a\}$. Let $W(t,x)=\upsilon(t,x)-V(t,x)=z(t,x)\varphi(t,y)$, one sees that $z(t,x)\geq0$ for  $t\leq\tau$, $x_1\in\{R,R+a\},y\in\R^{N-1}$. Since $\mathop{\sup}\limits_{t\in\R,x\in(R,R+a)\times\R^{N-1}}V(t,x)<+\infty$ and $\mathop{\inf}\limits_{t\leq\tau,x\in(R,R+a)\times\R^{N-1}}\upsilon(t,x)>0$,  there exists $t_0(a)\ll\tau$, which may depend on $a$ and sufficiently close to $-\infty$ such that $z(t_0(a),x)\geq0$ for $x\in(R,R+a)\times\R^{N-1}$. In addition that, we have
$$0\leq\dfrac{\mathcal{L}_\delta W}{\varphi}=\partial_tz-\Delta z-2\nabla z.\dfrac{\nabla\varphi}{\varphi}+\left(\lambda-\delta+\dfrac{c^2}{4}\right)z.$$
Since the zero order coefficient of parabolic operator with respect to $z$ is positive, we deduce from the parabolic weak maximum principle that $z(t,x)\geq0$ in $(t_0(a),\tau)\times\mathcal{Q}_{R+a}\setminus\mathcal{Q}_R$ for every $a>0$. Finally, the classical parabolic regularity implies that $V(\tau,x)\leq\upsilon(\tau,x)$ for $x\in\mathcal{Q}_{R+a}\setminus\mathcal{Q}_R$. Therefore,
$$U(\tau,x)\leq\lim_{a\to+\infty} \theta_{a}(x_1)\varphi(\tau,y)e^{-\frac{c}{2}x_1}= C(e^{(\sqrt{\kappa}+\sqrt{\kappa+\delta})R})e^{-(\sqrt{\kappa+\delta}+\frac{c}{2}) x_1}\quad\quad\textrm{$(C=\max_{\R\times\R^{N-1}}\varphi(t,y))$}.$$
The arbitrariness of $\tau$ enables us to conclude the proof.
\end{proof}
\begin{remark}In the proof of this theorem, we need not assume that the solution $U$ is  periodic in $y$ and nor in $t$, but the local regularity of solutions plays an important role. On the other hand, as seen from above, it is possible to choose $\kappa=\lambda+c^2/4-3\delta$ to obtain
$$U(\tau,x)\leq C_1e^{-(\sqrt{\lambda+\frac{c^2}{4}-2\delta}+\frac{c}{2}) x_1}\quad\textrm{for $\tau\in\R,x_1\geq R,y\in\R^{N-1}$}$$
Using the same arguments, we derive that there exists $C_2>0$ :
$$U(\tau,x)\leq C_2e^{(\sqrt{\lambda+\frac{c^2}{4}-2\delta}-\frac{c}{2}) x_1}\quad\textrm{for $\tau\in\R,x_1\leq-R,y\in\R^{N-1}$}.$$
Since $U$ is bounded, one can choose $C_1,C_2$ large enough such that these  inequalities hold in $\R^{N+1}$.
\end{remark}

\noindent{\bf Acknowledgements.} The research presented in this paper is a part of the PhD work. The author expresses his gratitude to professor Henri Berestycki for suggesting the problem and useful advices. He is also thankful to Luca Rossi for many interesting discussions. This work is supported by FIRST program (ITN-FP7/2007-2013), grant agreement 238702. He also thanks Technion-Israel Institute of Technology and Technische Universiteit Eindhoven for their encouragements, friendly and stimulating atmosphere when he visited. During the last year of the thesis, the author is supported by  ERC Grant
(FP7/2007-2013), grant agreement  321186: "Reaction-Diffusion Equations, Propagation and Modelling" held by Henri Berestycki. Lastly, he thanks the anonymous referees for the  helpful comments, which improve the paper.

%%%%%%%%%%%%%%%%%%%%%%%%%%%%%%%%%%%%%%%%%%%%%%%%%%%%%%%%%%%%%%%%%%
%%%%%%%%%%%%%%%%%%%%%%%%%%%%%%%%%%%%%%%%%%%%%%%%%%%%%%%%%%%%%%%%%%

\def\cprime{$'$} \def\polhk#1{\setbox0=\hbox{#1}{\ooalign{\hidewidth
  \lower1.5ex\hbox{`}\hidewidth\crcr\unhbox0}}}
  \def\cfac#1{\ifmmode\setbox7\hbox{$\accent"5E#1$}\else
  \setbox7\hbox{\accent"5E#1}\penalty 10000\relax\fi\raise 1\ht7
  \hbox{\lower1.15ex\hbox to 1\wd7{\hss\accent"13\hss}}\penalty 10000
  \hskip-1\wd7\penalty 10000\box7}

\end{document}